\documentclass[11pt,a4paper,reqno]{amsart}
\usepackage[T1]{fontenc}
\usepackage{lmodern}
\usepackage[margin=29mm]{geometry}
\usepackage{amsmath,amssymb,amsthm,mathtools}
\usepackage{microtype}
\usepackage{xcolor}
\usepackage[colorlinks=true,linkcolor=blue!45!black,
 citecolor=blue!45!black,urlcolor=blue!45!black]{hyperref}
\numberwithin{equation}{section}
\newtheorem{theorem}{Theorem}[section]
\newtheorem{proposition}[theorem]{Proposition}
\newtheorem{lemma}[theorem]{Lemma}
\newtheorem{corollary}[theorem]{Corollary}
\theoremstyle{definition}
\newtheorem{definition}{Definition}[section]
\newtheorem*{remark}{Remark}
\newcommand{\Chat}{\widehat{\mathbb C}}
\DeclareMathOperator{\Fill}{Fill}
\DeclareMathOperator{\Post}{Post}
\DeclareMathOperator{\Comp}{Comp}
\DeclareMathOperator{\modu}{mod}

\title{Wandering exposed Julia components of rational maps}
\author{Jinsong Zeng}
\date{}
\hypersetup{pdftitle={Wandering exposed Julia components of rational maps},
 pdfauthor={Jinsong Zeng}}
\begin{document}
\begin{abstract}
We prove that every wandering exposed Julia component of a rational
map is a point. Here a Julia component is exposed if it meets the
boundary of a Fatou domain. The theorem extends the corresponding result
for polynomials to rational maps, without any assumption on
critical recurrence. The proof combines an
enhanced nest with a covering estimate for multiply connected
pullback domains.
\end{abstract}

\maketitle

\section{Introduction}\label{sec:introduction}

We study the dynamics of rational maps $f:\Chat\to\Chat$ of degree
at least two. By Sullivan's theorem, every Fatou domain is
eventually periodic~\cite{Sullivan}. Julia components, however,
may wander. A Julia component is called \textbf{wandering} if its
forward images are pairwise distinct, and \textbf{exposed} if it
meets the boundary of a Fatou domain~\cite[Section~8.4]{PT}.
The main result of this paper is the following.

\begin{theorem}\label{thm:main}
Every wandering exposed Julia component of a rational map is a point.
\end{theorem}

For polynomials, all Julia components are exposed, since the Julia
set is the boundary of the basin of infinity. By the Branner--Hubbard
conjecture~\cite{BH}, proved independently by Kozlovski and
van Strien~\cite{KS} and Qiu and Yin~\cite{QY}, every wandering Julia
component of a polynomial is a point. Previously, Cui, Peng, and
Tan proved that every wandering exposed Julia component of a
rational map has connected complement after finitely many
iterates~\cite{CPT}. The conclusion of Theorem~\ref{thm:main} was
obtained in~\cite[Theorem~1.1]{GXY} under the assumption that every
wandering exposed Julia component containing a critical point is nonrecurrent.
We remove this assumption.

A Julia component is called \textbf{buried} if it is disjoint from
the boundary of every Fatou domain~\cite[\S7]{McMullen88}. Thus Theorem~\ref{thm:main} implies that every wandering
Julia component which is not a point is buried.
McMullen constructed hyperbolic rational maps with buried wandering
Jordan curves as Julia components~\cite{McMullen88}. Buried wandering
Jordan curves also occur as Julia components of cubic rational maps~\cite{Godillon}.

For geometrically finite rational maps, every wandering Julia
component is a point or a Jordan curve~\cite{PT}.
For cubic rational maps, without assuming geometric finiteness,
every wandering Julia component has at most two complementary
components after sufficiently many iterates~\cite{CPY}.
Exact annular systems have also been used to construct
postcritically finite rational maps whose connected Julia sets
contain wandering Jordan curves~\cite{CPT16}. Here the curves
are proper subsets of the unique Julia component.

To prove Theorem~\ref{thm:main}, it suffices to show that every wandering boundary component of an
infinitely connected invariant attracting or parabolic Fatou domain
is a point. Holomorphic models for such domains are known in the
attracting case~\cite{CP} and the parabolic case~\cite{GGP1,GGP2}.
We fix such a Fatou domain and construct a puzzle directly in
$\Chat$, using pullbacks of an absorbing domain. Our proof
does not use these holomorphic models. In the parabolic case, the
absorbing domain agrees with an attracting petal near the parabolic
fixed point.

We study the complementary components of this fixed Fatou domain
under the map induced by the action of $f$ on their boundaries.
First-entry and successor arguments, following~\cite{ZengFJ,RYZ}, reduce the proof
to persistent recurrence. In particular, two successive first
entries give pullbacks of bounded degree when the limit set of a
wandering complementary component contains a preperiodic component
or a wandering component with a strictly smaller limit set.
This argument also treats accumulation on the
component containing the parabolic fixed point, whether or not
that component is a point. The main difficulty is to prove complex
bounds for a nest of puzzle pieces in the persistently recurrent case.

We use the enhanced-nest construction of Kozlovski, Shen, and
van Strien~\cite{KSS} and its versions for puzzles
in~\cite{PQRTY,RY}. Related arguments for Newton maps appear
in~\cite{DS,WYZ,RYZ}. In our setting, the map on a surrounding
puzzle boundary may have degree one even when the enclosed
complementary component contains a critical point. In the nest
construction, we count visits to critical complementary components
even when the boundary degree is one. The nest annuli avoid the
forward orbits, under the induced map, of the critical components
in the limit set. Along the selected pullback chains,
the iterates of $f$ restrict to unbranched coverings of annuli.

The Kahn--Lyubich Covering Lemma, in its usual disk form, concerns
proper holomorphic maps between disks~\cite{KL}. In our setting,
a complementary component of the Fatou domain may contain a pole.
Although an iterate of $f$ maps one puzzle boundary onto another,
it need not map the enclosed piece properly onto the piece
bounded by the image curve. We use components of inverse images
under these iterates whose fillings
are the required pieces. These inverse components may be multiply
connected, so the disk form of the lemma does not apply directly.

We establish a covering lemma for proper holomorphic maps from
bounded finitely connected planar domains onto a disk
(Corollary~\ref{ana:physical-product}). Following Kahn and Lyubich,
we pass to a Galois branched covering and apply their Quasi-Additivity Law.
We then pass from the width estimate to a modulus estimate for
an annulus in the original plane. The smallness threshold depends
only on the total degree $D$. The coefficient is $2d^2$, where $d$
is the degree of the restriction to a chosen inverse component
of a smaller target disk. The nest and pullback estimates give
the required bounds for these two degrees and yield complex
bounds in the persistently recurrent case.

\noindent The paper is organized as follows. Section~\ref{dyn:setup}
gives the reduction, the construction of the puzzle, and the
relation between maps on puzzle boundaries and actual inverse components;
it also proves a modulus comparison for filled pullbacks.
Section~\ref{dyn:reduction} proves the bounded degree criterion
and reduces the problem to the persistently recurrent case.
Section~\ref{ana:kl-section} proves the covering estimate for
finitely connected planar domains. Section~\ref{comb:section} uses an
enhanced nest to prove complex bounds and complete the proof
of Theorem~\ref{thm:main}. Appendix~\ref{app:enhanced-nest}
gives the construction and proves the properties of the nest.

\section{Ends, puzzles, and proper pullbacks}\label{dyn:setup}

We reduce the theorem to an invariant attracting or parabolic
Fatou domain and construct a puzzle in $\Chat$. We then relate
the maps on piece boundaries to proper pullbacks and estimate
the degrees of these proper maps.

\begin{lemma}\label{dyn:basin-reduction}
Let $f$ be a rational map and $\Omega$ a Fatou domain of $f$.
Set $\Omega'=f(\Omega)$.
Then $f$ maps every component of $\partial\Omega$ onto a component
of $\partial\Omega'$.
Moreover, to prove Theorem~\ref{thm:main}, it suffices to show that every
wandering boundary component of an infinitely connected invariant
attracting or parabolic Fatou domain is a point.
\end{lemma}

\begin{proof}
We first prove the boundary correspondence. Let $E$ be a component
of $\partial\Omega$. Since $f(\partial\Omega)=\partial\Omega'$
and $f(E)$ is connected, $f(E)$ lies in a unique component $E'$
of $\partial\Omega'$.
We now prove that $E'\subset f(E)$.
Choose a Jordan neighborhood $U$ of the component of
$\Chat\setminus\Omega$ whose boundary is $E$,
with $\partial U\subset\Omega$. Choose a similar neighborhood $V$
of the component of $\Chat\setminus\Omega'$ whose boundary is $E'$,
with $\partial V\subset\Omega'$ and
$\overline V\cap f(\partial U)=\varnothing$.
Since $f:\Omega\to\Omega'$ is open and proper,
$f(\Omega\cap U)\cap V$ is nonempty, open and closed in the
connected set $\Omega'\cap V$. It therefore equals $\Omega'\cap V$.
Taking limits at points of $E'$ gives
\[
 E'\subset f(\overline U\cap\partial \Omega).
\]
Letting $\overline U$ decrease to the component of
$\Chat\setminus\Omega$ whose boundary is $E$, we obtain
$E'\subset f(E)$. Thus $f(E)=E'$.

Let $J$ be a wandering exposed Julia component. Choose a Fatou
domain $\Omega$ and a component $E$ of $\partial\Omega$ meeting $J$.
Since $E$ is connected
and lies in the Julia set, $E\subset J$. The connected set $J$
avoids $\Omega$, so it lies in the component of
$\Chat\setminus\Omega$ whose boundary is $E$. If $E$ is a point,
this complementary component, and hence $J$, is a point.
Moreover, $E$ is wandering, since its forward images lie in the
pairwise disjoint Julia components $f^n(J)$.

By Sullivan's theorem~\cite{Sullivan}, some iterate of $\Omega$
is a periodic Fatou domain. The corresponding image of $E$ remains
wandering under the first return map of this domain.

The periodic domain has infinitely many boundary components,
since otherwise none could wander. In particular, it is neither
a Siegel disk nor a Herman ring, which have one and two boundary
components, respectively. The classification of periodic Fatou
domains therefore leaves only attracting and parabolic
Fatou domains~\cite{Milnor}.

If an iterate of the original boundary component is a point, the
component lies in a finite fiber of that iterate of $f$. Being connected,
it is a point as well.
\end{proof}

\subsection{Ends and their dynamics}

From now on, we fix a rational map $f$ with an infinitely connected
invariant attracting or parabolic Fatou domain $\Omega$, and set
$R=\deg f$.
We choose coordinates so that $\infty\in\Omega$.
In the parabolic case, we denote by $p$ the parabolic fixed point
on $\partial\Omega$.

\begin{definition}\label{def:source-ends}
An \textbf{end} is a component of $\Chat\setminus\Omega$.
We denote by $\mathcal E$ the set of all ends. An end is
called \textbf{critical} if it contains a critical point of $f$.
We write $\mathcal E_{\mathrm{crit}}$ for this finite family.
\end{definition}

Every end $\mathbf e$ is a full continuum, and $\partial\mathbf e$
is a component of $\partial\Omega$. The end $\mathbf e$ is a point
if and only if $\partial\mathbf e$ is a point.

To describe the topology and dynamics of ends, we take the
quotient of $\Chat$ whose equivalence classes are the ends and
the singletons $\{z\}$ for $z\in\Omega$. We denote the quotient
map by $\pi$.

\begin{lemma}\label{new:end-realization}\label{dyn:ends}
The following statements hold.
\begin{enumerate}
\item The quotient space is a topological sphere, and
$\pi|_\Omega$ is a homeomorphism onto the open set $\pi(\Omega)$.
When working in the quotient sphere, we identify $\Omega$ with $\pi(\Omega)$.
\item The images of the ends under $\pi$ form a compact totally disconnected
subset of the quotient sphere. We identify $\mathcal E$ with
this subset and give the \textbf{end space} $\mathcal E$ its
subspace topology.
\item Each point of $\mathcal E$ has a neighborhood basis in the quotient sphere
consisting of Jordan disks whose boundaries lie in $\Omega$.
\item The map $f|_\Omega$ extends to an orientation-preserving
branched self-covering $F$ of the quotient sphere, of degree
$\deg(f|_\Omega)$, with
\[
 F^{-1}(\mathcal E)=\mathcal E.
\]
Moreover, $F\circ\pi=\pi\circ f$ on the closure of $\Omega$ in
the original sphere.
\end{enumerate}
\end{lemma}

\begin{proof}
\emph{The quotient sphere.}
The decomposition of $\Chat$ into ends and one-point sets in $\Omega$ is upper
semicontinuous: every Hausdorff limit of ends is connected and
lies in an end. Each decomposition element has connected complement,
so Moore's theorem~\cite{Moore} shows that the quotient is a topological sphere.
The restriction $\pi|_\Omega$ is a homeomorphism onto the open
set $\pi(\Omega)$.
From now on, we denote this quotient sphere by $\mathbb S^2$.

The set $\mathcal E=\mathbb S^2\setminus\pi(\Omega)$ is compact.
Since $\pi$ is closed and has connected fibers, the preimage
under $\pi$ of any continuum in $\mathbb S^2$ is a continuum in $\Chat$.
For a continuum in $\mathcal E$, this preimage lies in one end,
which $\pi$ collapses to a point. Thus every continuum in
$\mathcal E$ is a point, so $\mathcal E$ is totally disconnected.
A compact totally disconnected subset of a sphere
admits arbitrarily small Jordan neighborhoods of each of its
points with boundary disjoint from the set. This gives (3).
In particular, $\Omega$ is dense in $\mathbb S^2$.

\noindent
\emph{The map $F$.}
By Lemma~\ref{dyn:basin-reduction}, for each end $\mathbf e$ there
is a unique end $\mathbf e'$ with
$f(\partial\mathbf e)=\partial\mathbf e'$.
Define $F$ to agree with $f$ on $\Omega$ and to send the point
represented by $\mathbf e$ to the point represented by
$\mathbf e'$. Then $F^{-1}(\mathcal E)=\mathcal E$.
Since $\overline\Omega\cap\mathbf e=\partial\mathbf e$ in the
original sphere, this definition gives
$F\circ\pi=\pi\circ f$ on $\overline\Omega$.

\noindent
\emph{Continuity.}
Every fiber of $\pi$ meets $\overline\Omega$ in the original sphere,
since each end meets $\partial\Omega$. Thus
$\pi(\overline\Omega)=\mathbb S^2$. For a closed set
$E\subset\mathbb S^2$, the identity $F\circ\pi=\pi\circ f$ on
$\overline\Omega$ gives
\[
 F^{-1}(E)=\pi\bigl(\{z\in\overline\Omega:\pi(f(z))\in E\}\bigr).
\]
The set in braces is compact because $\pi\circ f$ is continuous.
Hence $F^{-1}(E)$ is compact and therefore closed, proving continuity.

\noindent
\emph{Openness.}
To prove openness at points of $\mathcal E$, fix $x\in\mathcal E$
and let $U\subset\mathbb S^2$ be any Jordan neighborhood of $x$
with $\partial U\subset\Omega$.
Choose a Jordan neighborhood $V$ of $F(x)$ with
$V\cap F(\partial U)=\varnothing$; this is possible because
$F(x)\in\mathcal E$ and $F(\partial U)\subset\Omega$.
We will show that $V\subset F(U)$.
Choose $x_0\in U\cap\Omega$ with $f(x_0)\in V$.

We next obtain preimages in $U\cap\Omega$.
For any $y\in V$, take $y_n\in V\cap\Omega$ with $y_n\to y$.
Since $V\cap\Omega=V\setminus\mathcal E$ is connected, join
$f(x_0)$ to $y_n$ by a path in $V\cap\Omega$, and lift this
path under the proper branched covering $f:\Omega\to\Omega$,
starting at $x_0$.
The lift cannot cross $\partial U$, since its image lies in $V$
and $V\cap F(\partial U)=\varnothing$.
Its endpoint therefore satisfies $x_n\in U\cap\Omega$ and
$f(x_n)=y_n$.

We now pass to a limit to prove that $y\in F(U)$.
By compactness of $\overline U$, a subsequence of $x_n$ converges
to $x'\in\overline U$. Continuity gives $F(x')=y$.
Since $y\notin F(\partial U)$, we have $x'\in U$.
Thus $V\subset F(U)$. Such Jordan neighborhoods $U$ form a
neighborhood basis at $x$, so $F$ is open at $x$.
The map $F$ is already open on $\Omega$, hence is open on
$\mathbb S^2$.

\noindent
\emph{The branched covering.}
Set $d=\deg(f|_\Omega)$. If $\mathbf e\in\mathcal E$ had $d+1$
distinct preimages under $F$, choose pairwise disjoint neighborhoods
of these preimages in $\mathbb S^2$. By openness of $F$ and density
of $\Omega$, the intersection of their images under $F$ contains a
regular value of $f|_\Omega$. Since $F^{-1}(\Omega)=\Omega$,
this value would have at least $d+1$ preimages under $f|_\Omega$,
a contradiction. Thus every fiber of $F$ is finite and hence
totally disconnected. Since $F^{-1}(\Omega)=\Omega$ and $F=f$ on
$\Omega$, a regular value of $f|_\Omega$ has exactly $d$ preimages
under $F$, each of local degree one. Hence $\deg F=d>0$.
By~\cite[Theorem~A.15]{BM}, the continuous open map $F$ is therefore
an orientation-preserving branched covering.
\end{proof}

\begin{definition}\label{def:end-dynamics}
Let $F:\mathbb S^2\to\mathbb S^2$ be the branched covering given
by Lemma~\ref{new:end-realization}.
\begin{enumerate}
\item Define $\sigma:\mathcal E\to\mathcal E$ by setting
$\sigma(\mathbf e)=\mathbf e'$ when $F$ sends the point represented
by $\mathbf e$ in $\mathbb S^2$ to the point represented by
$\mathbf e'$. Thus
\[
 f(\partial\mathbf e)=\partial\sigma(\mathbf e).
\]
By Lemma~\ref{new:end-realization}, $\sigma$ is continuous, and each
end has at most $\deg(f|_\Omega)$ preimages under $\sigma$.
\item An end $\mathbf e$ is called \textbf{periodic} if
$\sigma^k(\mathbf e)=\mathbf e$
for some $k\ge1$, and \textbf{preperiodic} if some iterate under
$\sigma$ is periodic. It is called \textbf{wandering} if all its
forward images under $\sigma$ are distinct.
\end{enumerate}
\end{definition}

For a bounded connected set $X\subset\mathbb C$, let $\Fill(X)$
be the union of $X$ and all bounded components of
$\mathbb C\setminus X$. Thus $\Fill(\partial\mathbf e)=\mathbf e$
for every end $\mathbf e$. If $X$ is a bounded domain with finitely
many Jordan boundary curves, $\Fill(X)$ is the open disk bounded
by its outer boundary.

Recall that $R=\deg f$.

\begin{lemma}\label{dyn:coverage}
There are at most $2R-2$ critical ends. Every pole outside $\Omega$
belongs to a critical end. If $\mathbf e$ is not critical, then $f$
is univalent on a neighborhood of $\mathbf e$ and maps $\mathbf e$
onto $\sigma(\mathbf e)$. Every branch point of $F$ in
$\mathcal E$ corresponds to a critical end.
\end{lemma}

\begin{proof}
There are at most $2R-2$ critical points, and hence at most
$2R-2$ critical ends.

Let $\mathbf e$ be a noncritical end. Using Jordan neighborhoods
as in the proof of Lemma~\ref{dyn:basin-reduction}, choose bounded Jordan disks
$V_j\subset\mathbb C$, with boundaries in $\Omega$, whose nested closures
decrease to $\sigma(\mathbf e)$ and avoid $f(\infty)$.
Let $U_j$ be the component of $f^{-1}(V_j)$ containing
$\partial\mathbf e$.

The continuum $\bigcap_j\overline{U_j}$ contains
$\partial\mathbf e$ and is disjoint from $\Omega$, since
$f(\Omega)\subset\Omega$; hence it lies in $\mathbf e$.
For large $j$, $U_j$ therefore contains no critical point of $f$, so
the proper map $f:U_j\to V_j$ is an unbranched covering.
By the Riemann--Hurwitz formula, $U_j$ is a disk and
$f:U_j\to V_j$ has degree one. The bounded disk $U_j$ contains
$\partial\mathbf e$, hence also $\mathbf e$.
Taking intersections gives $f(\mathbf e)=\sigma(\mathbf e)$.
Thus $f$ is univalent near $\mathbf e$, and $\mathbf e$ contains
no pole because $V_j$ is bounded. The corresponding point of
$\mathbb S^2$ is not a branch point of $F$.
\end{proof}

\begin{remark}
A critical end need not give a branch point of $F$.
The local degree of $F$ at the corresponding point is determined
by $f|_\Omega$ near the boundary of the end, whereas criticality
of the end also counts
critical points of $f$ in its interior.
\end{remark}

\subsection{The absorbing domain and its puzzle}

A domain $W\subset\Omega$ is \textbf{absorbing} if $f(W)\subset W$
and every compact subset of $\Omega$ is mapped into $W$ by all
sufficiently large iterates of $f$. We construct the puzzle by
successive pullbacks of one such domain, chosen so that all its
pullbacks in $\Omega$ are connected.

\begin{lemma}\label{new:absorbing-core}
There is an absorbing domain $W\subset\Omega$ containing
$\infty$ and all critical points and poles in $\Omega$.
Its boundary consists of finitely many Jordan curves. The part
$\partial W\cap\Omega$ contains no postcritical point of $f$.
In the attracting case, $\overline W\subset\Omega$ and
$f(\overline W)\subset W$. In the parabolic case,
\[
  \overline W\setminus\Omega=\{p\},\qquad
  f(\overline W)\subset W\cup\{p\},
\]
and $W$ agrees with an attracting petal near $p$.
The point $p$ may lie on a forward critical orbit. For $n\ge0$, set
$W_n=f^{-n}(W)\cap\Omega$. Then each $W_n$ is connected, and every
component of $\Chat\setminus W_n$ contains an end.
\end{lemma}

\begin{proof}
In the attracting case, choose a Jordan disk $U$ about the fixed
point such that $\overline U\subset\Omega$ and
$f(\overline U)\subset U$. In the parabolic case, choose a Jordan
attracting petal $U\subset\Omega$ such that
\[
 \overline U\setminus\Omega=\{p\},\qquad
 f(\overline U)\subset U\cup\{p\};
\]
see~\cite{Milnor}. In either case, choose $\partial U\cap\Omega$
to avoid the postcritical points of $f$. The domain $U$ absorbs every compact
subset of $\Omega$.

Let $U_k$ be the component of $f^{-k}(U)$ containing $U$.
Then these domains are nested and exhaust $\Omega$. Join $\infty$, all points of
$f^{-1}(\infty)\cap\Omega$, and all critical points in $\Omega$
to a chosen point of $U$ by finitely many arcs in $\Omega$.
Denote the union of these arcs by $E$.
Fix $k\ge1$ so large that $E\subset U_k$. We have
\[
 f(\overline{U_k})\subset\overline{U_{k-1}},\qquad
 \overline{U_{k-1}}\cap\Omega\subset U_k.
\]

In the attracting case, $\partial U_k$ consists of disjoint Jordan
curves. In the parabolic case, $\partial U_k$ is locally a
simple arc except possibly at the finitely many points of
$\overline{U_k}\setminus\Omega\subset f^{-k}(p)$.
Fill every component of $\Chat\setminus U_k$ contained in
$\Omega$, and denote the resulting domain by $V$. Each component
of $\Chat\setminus V$ contains an end.

We show that $f(B)\subset V$ for every component $B$ of
$\Chat\setminus U_k$ contained in $\Omega$.
Since $f(\partial B)\subset U_k$, the intersection of $f(B)$ with
any component of $\Chat\setminus V$ is both open and closed in
that component. If this intersection were nonempty, the whole
component would lie in $\Omega$, contrary to its containing an end.
Thus
\[
 f(\overline V)\cap\Omega\subset V.
\]
In the attracting case, set $W=V$. Then
$\overline W\subset\Omega$ and $f(\overline W)\subset W$.

In the parabolic case, we have
$\overline V\setminus\Omega
=\overline{U_k}\setminus\Omega\subset f^{-k}(p)$.
The finite set $\overline V\setminus\Omega$ is forward invariant.
We remove small neighborhoods
of its points other than $p$, while keeping $U\cup E$.

Within an open neighborhood of $\partial\Omega$ disjoint from $E$,
choose pairwise disjoint small closed Jordan neighborhoods $\mathcal N_z$
of the points $z\in\overline V\setminus\Omega$, $z\ne p$,
avoiding $\overline U$. Choose $\mathcal N_z$ so that $V\cap\operatorname{int}\mathcal N_z$
has finitely many components, each a Jordan domain bounded by two
arcs of $\partial V$ meeting at $z$ and one arc of $\partial\mathcal N_z$.
The points $z$ may be critical or lie on a forward critical orbit.

We choose the neighborhoods $\mathcal N_z$ in decreasing order of the least
integer $j\ge1$ for which $f^j(z)=p$. At each step, take $\mathcal N_z$
small enough that
\[
 f^{-1}(\mathcal N_z)\cap\overline V
 \subset
 \bigcup_{\substack{z'\in\overline V\setminus\Omega\\ f(z')=z}}
       \operatorname{int}\mathcal N_{z'}.
\]
We also choose $\partial\mathcal N_z\cap V$ to avoid the postcritical
points of $f$. Set
\[
 W=V\setminus\bigcup_{z\ne p}\mathcal N_z,
\]
where the union is over the points just considered.
Then $W$ is connected, contains $U\cup E$, and has finitely many
Jordan boundary curves. The set $\partial W\cap\Omega$ contains
no postcritical point of $f$. Near $p$, $\partial W$ agrees with $\partial V$, and
$W$ agrees with an attracting petal.
Each $\mathcal N_z$ meets $\Chat\setminus V$
at $z$, so every component of $\Chat\setminus W$ contains a
component of $\Chat\setminus V$, and hence an end.

By the choice of the $\mathcal N_z$, $f(\overline W)$ avoids every $\mathcal N_z$.
Hence
\[
 \overline W\setminus\Omega=\{p\},\qquad
 f(\overline W)\subset W\cup\{p\}.
\]

In both cases, $E\subset W\subset f^{-1}(W)\cap\Omega$.
Every component of $f^{-1}(W)\cap\Omega$ maps properly onto $W$
under $f$, hence meets $f^{-1}(\infty)\cap\Omega\subset E$.
Since $W$ is connected, there is only one such component.
The same argument shows inductively that every $W_n$ is connected:
at each step, $W_n$ is forward invariant and contains $E$.
The sets $W_n$ increase and exhaust $\Omega$.

In the quotient sphere of Lemma~\ref{new:end-realization},
$W_n=F^{-n}(W)$. Hence $F^n$ maps each component of
$\mathbb S^2\setminus W_n$ onto a component of
$\mathbb S^2\setminus W$, which meets $\mathcal E$.
Since $F^{-n}(\mathcal E)=\mathcal E$, each component of
$\mathbb S^2\setminus W_n$ also meets $\mathcal E$.
Since $\pi$ is closed and has connected fibers, it maps each
component $X$ of $\Chat\setminus W_n$ onto a component
$\pi(X)$ of $\mathbb S^2\setminus W_n$, with
$X=\pi^{-1}(\pi(X))$. Since $\pi(X)$ meets $\mathcal E$,
$X$ contains an end.
\end{proof}

From now on, we fix $W$ as in Lemma~\ref{new:absorbing-core} and
build the puzzle from its pullbacks $W_n$.

\begin{definition}\label{def:puzzle}
A \textbf{(puzzle) piece} of depth $n$ is a component of
$\Chat\setminus\overline{W_n}$.
A \textbf{depth-$\boldsymbol{n}$ component} is a component of
$\Chat\setminus W_n$.
In the parabolic case, a \textbf{preparabolic vertex} of depth $n$ is a point
$z\in\partial W_n\setminus\Omega$. It satisfies $f^n(z)=p$ and may
be critical or lie on a forward critical orbit.
\end{definition}

Pieces are open Jordan disks. For each end $\mathbf e$ and depth
$n$, denote by $X_n(\mathbf e)$ the unique depth-$n$ component
containing $\mathbf e$. If $\mathbf e$ is contained in a piece
of depth $n$, denote that piece by $P_n(\mathbf e)$.
In the attracting case, $P_n(\mathbf e)$ always exists and
$X_n(\mathbf e)=\overline{P_n(\mathbf e)}$.
In the parabolic case, $X_n(\mathbf e)$ is a finite union of piece
closures, which meet only at preparabolic vertices. At a
depth-$n$ vertex $z$, at most
$\deg_z(f^n)$ pieces meet; we do not need this bound below.

\begin{lemma}\label{dyn:puzzle}\label{new:global-cylinders}
For every $n\ge0$, the depth-$n$ components have the following
properties.
\begin{enumerate}
\item There are finitely many components, each a full continuum.
Every end is contained in exactly one depth-$n$ component.
\item For every end $\mathbf e$, the sets $X_n(\mathbf e)$ satisfy
\[
 X_{n+1}(\mathbf e)\subset X_n(\mathbf e),\qquad
 \bigcap_{n\ge0}X_n(\mathbf e)=\mathbf e.
\]
If $\mathbf a\ne\mathbf e$ is another end, then
$X_n(\mathbf e)\cap X_n(\mathbf a)=\varnothing$ for all sufficiently
large $n$.
\item For every end $\mathbf e$ and $r\ge1$, the map $\sigma^r$
takes the ends in $X_{n+r}(\mathbf e)$ onto the ends in
$X_n(\sigma^r(\mathbf e))$. Moreover,
\[
 f^r\bigl(\partial X_{n+r}(\mathbf e)\bigr)
 =\partial X_n(\sigma^r(\mathbf e)).
\]
\item If $X_{n+1}(\mathbf e)$ contains no critical end,
then $f$ is univalent on a neighborhood of $X_{n+1}(\mathbf e)$ and
$f(X_{n+1}(\mathbf e))=X_n(\sigma(\mathbf e))$.
\end{enumerate}
\end{lemma}

\begin{proof}
For (1), $W_n$ is a component of $f^{-n}(W)$ and is therefore
finitely connected. Its complementary components are finitely
many full continua. Every end is connected and disjoint from
$W_n$, so it lies in exactly one of these components.

For (2), the sets $W_n$ increase and exhaust $\Omega$. Hence the
$X_n(\mathbf e)$ are nested, and their intersection is a connected
subset of $\Chat\setminus\Omega$ containing $\mathbf e$.
This intersection therefore equals $\mathbf e$. Since distinct
ends are disjoint, compactness gives the separation assertion in (2).

For (3), $W_{n+r}=F^{-r}(W_n)$ in the quotient sphere.
Thus $F^r$ maps $\pi(X_{n+r}(\mathbf e))$ onto
$\pi(X_n(\sigma^r(\mathbf e)))$. The identity
$F^{-r}(\mathcal E)=\mathcal E$ gives the correspondence between ends.
The proper map $f^r:W_{n+r}\to W_n$ maps each boundary component
of $W_{n+r}$ onto a boundary component of $W_n$.
On $\partial X_{n+r}(\mathbf e)\subset\overline\Omega$, we have
$\pi\circ f^r=F^r\circ\pi$. The image relation above therefore gives
\[
 f^r(\partial X_{n+r}(\mathbf e))
 \subset\pi^{-1}\bigl(\pi(X_n(\sigma^r(\mathbf e)))\bigr)
 =X_n(\sigma^r(\mathbf e)).
\]
Thus the image boundary component is
$\partial X_n(\sigma^r(\mathbf e))$.

For (4), set $X=X_{n+1}(\mathbf e)$ and $Y=X_n(\sigma(\mathbf e))$.
Every end in $X$ is noncritical. By Lemma~\ref{dyn:coverage} and
$W_{n+1}=f^{-1}(W_n)\cap\Omega$, the continuum $f(X)$ avoids
$W_n$ and contains $\sigma(\mathbf e)$, so $f(X)\subset Y$.
The same preimage identity shows that $X$ is a component of
$f^{-1}(Y)$. Since all critical points in $\Omega$ lie in $W$,
$X$ contains no critical point of $f$.
The Jordan neighborhood argument in Lemma~\ref{dyn:coverage}
now shows that $f$ is univalent near $X$ and maps $X$ onto $Y$.
\end{proof}

For each fixed $n$, the images under $\pi$ of the depth-$n$
components form a finite family of pairwise disjoint compact sets.
Thus, if a sequence of ends converges to $\mathbf e$ in $\mathcal E$,
all sufficiently late terms lie in $X_n(\mathbf e)$.

In the parabolic case, we denote the end containing $p$ by
$\mathbf e_p$ and call it the \textbf{parabolic end}. An end whose
orbit avoids $\mathbf e_p$ meets no preparabolic vertex, so it lies
in a unique piece at every depth. This includes every wandering
end. For any such end $\mathbf e$ and each $n$, some $m>n$ satisfies
\[
 \overline{P_m(\mathbf e)}\subset P_n(\mathbf e),
\]
since $X_m(\mathbf e)$ decreases to $\mathbf e\subset P_n(\mathbf e)$.
In the attracting case, the same statement holds for every end.

Every piece $P$ is nice: $f^k(\partial P)\cap P=\varnothing$
for $k\ge1$. If $P$ has depth $n$, then
$f^k(\partial P\cap\Omega)\subset\partial W_{n-k}$ for $0\le k\le n$,
and $f^{n+1}(\partial P\cap\Omega)\subset W$. In the parabolic case,
every vertex of $\partial P$ maps to $p$ under $f^n$; in particular,
no piece contains $p$.

\begin{lemma}\label{new:unmarked-closed}
For every $n\geq0$, let $P$ be a piece of depth $n+1$
whose boundary meets no
critical end. Then there is a unique piece $Q$ of depth $n$ such
that $f(\partial P)=\partial Q$. The map $f$ sends points of $P$
sufficiently close to $\partial P$ into $Q$.
Moreover, $f(\Omega\cap P)=\Omega\cap Q$, and
$\sigma(\mathbf e)\subset Q$ for every end $\mathbf e\subset P$.
\end{lemma}

\begin{proof}
The closure of each component of $f^{-1}(\Omega)$ other than
$\Omega$ is a continuum in $\Chat\setminus\Omega$, and hence
lies in an end. This end must be critical, since a noncritical
end maps into $\Chat\setminus\Omega$ by Lemma~\ref{dyn:coverage}.
Since $\partial P$ meets no critical end, we can choose an inner
annular neighborhood $A\subset P$ of $\partial P$, with a Jordan inner boundary,
disjoint from the closures of the components of $f^{-1}(\Omega)$
other than $\Omega$. Then
\[
 f^{-1}(\Omega)\cap A=\Omega\cap A,\qquad
 f^{-1}(W_n)\cap A=W_{n+1}\cap A=\varnothing.
\]
The set $f(A)$ is open and disjoint from $W_n$, so it is also
disjoint from $\overline{W_n}$. Since $f(A)$ is connected,
it lies in a single piece $Q$ of depth $n$.

Approaching $\partial P$ from $A$ and from $W_{n+1}$ gives, by continuity,
\[
 f(\partial P)\subset\overline Q\cap\overline{W_n}=\partial Q.
\]

By the choice of $W$ and Lemma~\ref{dyn:coverage}, the map $f$
is locally univalent on $\partial P$. Since $\partial P$ and
$\partial Q$ are Jordan curves, $f(\partial P)$ is relatively open
in $\partial Q$. It is also closed by compactness, so
$f(\partial P)=\partial Q$.

Now join any point of $\Omega\cap P$ to $\infty$ by a path in
$\Omega$ and stop at its first crossing of $\partial P$.
Before the crossing, the image path avoids $\partial W_n$
and lies in $Q$ near the crossing. Hence $f(\Omega\cap P)\subset Q$.

Conversely, fix $y\in\Omega\cap Q$ and take a path from $y$ to
$\infty$ in $\Omega$. Reverse the segment up to its first intersection
with $\partial Q$; the resulting path starts on $\partial Q$ and
otherwise lies in $\Omega\cap Q$. Since $f(\partial P)=\partial Q$,
choose a preimage of its starting point in $\partial P\cap\Omega$.
Lift the path from this preimage under the proper map
$f:\Omega\to\Omega$. Local univalence on $\partial P$ and
$f(A)\subset Q$ show that the lift enters $P$.
It cannot meet $\partial P$ again, since its image lies in $Q$
after the starting point. Its endpoint therefore lies in
$\Omega\cap P$ and maps to $y$. Thus $f(\Omega\cap P)=\Omega\cap Q$.

If $\mathbf e\subset P$, this equality
and the openness of $F$ in Lemma~\ref{new:end-realization} give
$f(\partial\mathbf e)=\partial\sigma(\mathbf e)\subset Q$.
Since $Q$ is a Jordan disk, $\sigma(\mathbf e)\subset Q$.
\end{proof}

\begin{remark}
In the attracting case, $\partial P\subset\Omega$, so the condition
on $\partial P$ holds. In the parabolic case, a
regular preparabolic vertex may belong to an end containing a
critical point elsewhere. Excluding critical ends from $\partial P$
ensures that a neighborhood of $\partial P$ meets no component of
$f^{-1}(\Omega)$ other than $\Omega$. This allows us to choose
$A$ with $f(A)\cap W_n=\varnothing$, as required above.

Lemma~\ref{dyn:puzzle}(3) concerns the whole depth-$(n+1)$ component,
which may contain several piece closures in the parabolic case.
Its boundary correspondence follows from the proper map
$f:W_{n+1}\to W_n$ and requires no condition on
critical ends.
\end{remark}

Lemma~\ref{new:unmarked-closed} allows us to define $\sigma(P)$
and provides the inner annular neighborhood needed for the proper
pullbacks in Lemma~\ref{ana:exact-hull}.

\begin{definition}\label{def:puzzle-chain}
For $n\ge0$, let $P$ be a piece of depth $n+1$ whose boundary meets no
critical end. Define $\sigma(P)=Q$, where $Q$ is the unique piece
of depth $n$ with $f(\partial P)=\partial Q$, as given by
Lemma~\ref{new:unmarked-closed}.

For an integer $r\ge1$, a \textbf{(puzzle) chain} of
length $r$ is a sequence of pieces $P_{n+r},\ldots,P_n$, with
$P_j$ of depth $j$, such that $\partial P_j$ meets no critical
end and $\sigma(P_j)=P_{j-1}$ for $n<j\le n+r$.
We write $P_{n+r}\xrightarrow{r}P_n$ for this chain.
This notation does not assert $f^r(P_{n+r})=P_n$.
The \textbf{initial piece} and \textbf{terminal piece} are
$P_{n+r}$ and $P_n$, respectively. The \textbf{source pieces}
are $P_{n+r},\ldots,P_{n+1}$, each the initial piece of one step.
A \textbf{subchain} consists of one or more consecutive steps of a chain.
\end{definition}

\begin{lemma}\label{ana:chain-existence}
Let $Q$ be a piece of depth $n$, and let $\mathbf e$ be an end
such that $\sigma^r(\mathbf e)\subset Q$, where $r\ge1$.
Suppose that
\[
 \partial Q\cap\sigma^k(\mathbf c)=\varnothing
 \quad\text{for every }\mathbf c\in\mathcal E_{\mathrm{crit}}
 \text{ and }1\le k\le r.
\]
Then $P_{n+r}(\mathbf e)$ exists and
$P_{n+r}(\mathbf e)\xrightarrow{r}Q$ is a chain.
\end{lemma}

\begin{proof}
We first consider $r=1$. If $\mathbf e$ met $\partial W_{n+1}$,
then $\sigma(\mathbf e)$ would meet $\partial W_n$, contrary to
$\sigma(\mathbf e)\subset Q$. Thus $P=P_{n+1}(\mathbf e)$ exists.

We next show that $f(\partial P)=\partial Q$.
By the condition on $\partial Q$, the set
$f^{-1}(\partial Q)\cap\overline\Omega$ meets no critical end.
Each component of this set is a Jordan curve in $\partial W_{n+1}$,
bounds a piece, and maps onto $\partial Q$ under $f$.
Choose $z\in\Omega\cap P$ sufficiently close to
$\partial\mathbf e$ that $f(z)\in Q$, and join $z$ to $\infty$
by a path in $\Omega$. Before its first exit from $P$, the path
avoids $\partial W_{n+1}$, so its image under $f$ avoids
$\partial W_n$ and remains in $Q$. At the exit, the path meets a
component of $f^{-1}(\partial Q)\cap\overline\Omega$, which
is therefore $\partial P$. Thus $f(\partial P)=\partial Q$ and
$\partial P$ meets no critical end. By
Definition~\ref{def:puzzle-chain}, $\sigma(P)=Q$, so
$P\xrightarrow{1}Q$ is a chain.

For $r>1$, apply the one-step case with target $Q$ to
$\sigma^{r-1}(\mathbf e)$. The boundary of the resulting piece
avoids $\sigma^k(\mathbf c)$ for every critical end $\mathbf c$
and $1\le k<r$, since $f$ maps this boundary onto $\partial Q$.
Use this piece as the next target and continue backward along
the orbit of $\mathbf e$ to obtain the required chain.
\end{proof}

\begin{remark}
In the attracting case, the boundary condition holds because
$\partial Q\subset\Omega$. In the parabolic case, only critical
ends eventually mapping to $\mathbf e_p$
need to be considered. The ends in their positive forward orbits form a
finite family. By Lemma~\ref{dyn:puzzle}(2), a sufficiently deep
component about any end outside this family is disjoint from the family;
every piece whose closure lies in that component satisfies
the boundary condition for all $r\ge1$.
\end{remark}

\subsection{Proper pullbacks, degrees, and moduli}
\label{ana:physical-section}\label{dyn:real-section}

For planar sets $E\subset T\subset\mathbb C$ with $E$ contained
in a single component of $T$, we denote that component by $\Comp_E(T)$.
Recall that $\Fill(S)$ is $S$ together with all bounded components
of $\mathbb C\setminus S$. The next lemma identifies the filled
pullback along an end with the initial piece of its chain.

\begin{lemma}\label{ana:exact-hull}
Let $P\xrightarrow{r}Q$ be a chain and let
$\mathbf e\subset P$ satisfy $\sigma^r(\mathbf e)\subset Q$. Then
\[
 S=\Comp_{\partial\mathbf e}(f^{-r}(Q))
\]
is finitely connected, $f^r:S\to Q$ is proper, and $\Fill(S)=P$.
\end{lemma}

\begin{proof}
Applying Lemma~\ref{new:unmarked-closed} along the chain gives
\[
 f^r(P\cap\Omega)=Q\cap\Omega,
 \qquad f^r(\partial P)=\partial Q,
\]
and an inner annular neighborhood $A\subset P$ of $\partial P$
with $f^r(A)\subset Q$. Since $\Omega$ is connected and contains
$\infty$, every component of $P\cap\Omega$ meets $A$.
Thus $A\cup(P\cap\Omega)$ lies in one component of $f^{-r}(Q)$.
Since $\partial\mathbf e\subset\overline{P\cap\Omega}\cap f^{-r}(Q)$,
this component is $S$.

Since $f^r(\partial P)=\partial Q$, the component $S$ cannot
cross $\partial P$. Thus $A\subset S\subset P$, which gives
$\Fill(S)=P$.
The map $f^r:S\to Q$ is proper, and $S$ is finitely connected
because $S$ is a component of $f^{-r}(Q)$ and $Q$ is a Jordan disk.
\end{proof}

In the parabolic case, inner boundary components of $S$ may contain
critical points of $f^r$ mapping to vertices of $\partial Q$.
We therefore do not assume that all boundary components of $S$
are Jordan curves. The covering
estimate in Corollary~\ref{ana:physical-product} also applies to
these domains.

For a chain $P\xrightarrow{r}Q$, a \textbf{critical-end visit}
is a step whose source piece contains a critical end.
We denote the number of such visits by
\[
 \nu(P\to Q)
 =\#\{0\le j<r:\sigma^j(P)\text{ contains a critical end}\}.
\]
Each such step is counted once, even if its source piece contains
several critical ends. The terminal piece $Q$ is not counted.

\begin{lemma}\label{ana:real-degree}
If $\nu(P\to Q)\le L$, then the component
$S$ in Lemma~\ref{ana:exact-hull} satisfies
\begin{equation*}
 \deg(f^r:S\to Q)\le R^L.
\end{equation*}
The same bound holds for each subchain
\[
 \sigma^j(P)\xrightarrow{r'}\sigma^{j+r'}(P),
 \qquad 1\le r'<r,\quad 0\le j\le r-r'.
\]
\end{lemma}

\begin{proof}
For $0\le j<r$, the set $f^j(S)$ is a component of
$f^{-(r-j)}(Q)$. Applying Lemma~\ref{ana:exact-hull} to the
remaining chain gives $f^j(S)\subset\sigma^j(P)$.
If $\sigma^j(P)$ contains no critical end, the component of
$f^{-1}(\sigma^{j+1}(P))$ containing $f^j(S)$ lies in
$\sigma^j(P)$ and has no critical point of $f$.
Its proper map under $f$ onto $\sigma^{j+1}(P)$
is an unbranched covering of a disk and hence conformal.
In particular, $f:f^j(S)\to f^{j+1}(S)$ has degree one.
At every other step the degree is at most $R$. Multiplying the
degrees gives the bound. Each subchain above has at most
$L$ critical-end visits, so the same argument applies.
\end{proof}

Throughout the paper, the degree of a chain $P\xrightarrow{r}Q$
means $\deg(f^r:S\to Q)$, where $S$ is the component of
$f^{-r}(Q)$ containing $P\cap\Omega$, as in
Lemma~\ref{ana:exact-hull}. The boundary degree of
$f^r:\partial P\to\partial Q$ can be smaller than the proper degree.

We next give a modulus estimate for filled pullbacks.
For a Jordan disk $B$ compactly contained in a simply connected
domain $Q$, write
\[
 \mathbb A(B,Q)=Q\setminus\overline B,
 \qquad \modu(B,Q)=\modu\mathbb A(B,Q).
\]

\begin{lemma}\label{ana:filled-pullback}
Let $G:S\to Q$ be a proper holomorphic map of degree $D$, where
$S$ is a bounded planar domain and $Q$ is a Jordan disk. Let
$B\Subset Q$ be a Jordan disk and $A$ a component of $G^{-1}(B)$.
Suppose that $\hat A=\Fill(A)$ is a Jordan disk, and set
$\hat S=\Fill(S)$. Then
\[
 \modu(B,Q)\le D\modu(\hat A,\hat S).
\]
\end{lemma}

\begin{proof}
Since $G$ is proper, $A\Subset S$. Then $\hat A\Subset\hat S$.
First suppose that $\hat S$ is a Jordan disk and that $G$ extends
continuously to $\overline S$.
Write $\textup{cl}(\hat A)$ for the closure of $\hat A$.
Let $\Gamma$ be the family of curves joining $\partial A$ to
$\partial S$ with interiors in $S\setminus\overline A$.
Let $\mathcal F$ be the vertical foliation of
$\hat S\setminus\textup{cl}(\hat A)$ obtained from the vertical
segments in a conformally equivalent flat cylinder
\cite[Example~4.1]{KL}. Orient each leaf from $\partial\hat A$
to $\partial\hat S$ and stop at its first intersection with
$\partial S$. Since $\partial\hat A\subset\partial A\subset S$,
this initial segment belongs to $\Gamma$. Thus $\mathcal F$
overflows $\Gamma$ in the sense of~\cite[\S4.2]{KL}.

Write $G(\Gamma)=\{G\circ\gamma:\gamma\in\Gamma\}$ for the family
of image curves. Every curve in $G(\Gamma)$ contains a subcurve joining
$\partial B$ to $\partial Q$ through $Q\setminus\overline B$.
Thus $G(\Gamma)$ overflows the family of curves crossing this
annulus. Write $\mathcal L$ for extremal length. Monotonicity
under overflow and~\cite[Lemma~4.3]{KL} give
\[
 \modu(B,Q)\le\mathcal L(G(\Gamma))
 \le D\mathcal L(\Gamma)
 \le D\mathcal L(\mathcal F)
 =D\modu(\hat A,\hat S).
\]

For a general proper map, exhaust $Q$ by smooth Jordan disks
$Q_j$ containing $\overline B$ and whose boundaries avoid the
critical values of $G$. Let $S_j$ be the component of $G^{-1}(Q_j)$
containing $A$. Then $S_j$ increases to $S$, and
$\deg(G:S_j\to Q_j)\le D$.
Apply the preceding argument to $G:S_j\to Q_j$ and pass to the
limit, using the fact that $\Fill(S_j)$ increases to $\hat S$ and the
continuity of modulus.
\end{proof}

\subsection{Annuli free of critical ends}

In the nest construction, we need the pullback of an annulus along
a chain to be an annulus, and the corresponding iterate of $f$ to
give an unbranched holomorphic covering onto the original annulus.
The next lemma gives a condition on critical ends that ensures this
and proves that the inner puzzle piece exists and is unique.

\begin{lemma}\label{comb:complete-pullback}
Let $P_{n+r}\xrightarrow{r}P_n$ be a chain and let
$Q\Subset P_n$ be a piece. Suppose the following
\textbf{critical-end condition} holds:
\[
 \sigma^j(\mathbf c)\subset Q
 \quad\text{whenever }1\le j\le r,\quad
 \mathbf c\in\mathcal E_{\mathrm{crit}},\quad \mathbf c\subset P_{n+j}.
\]
Then exactly one piece $P\subset P_{n+r}$ satisfies $\sigma^r(P)=Q$,
and $P\Subset P_{n+r}$. For $0\le j<r$, the annulus
$\mathbb A(\sigma^j(P),P_{n+r-j})$ contains no critical end, and
\[
 f:\mathbb A(\sigma^j(P),P_{n+r-j})
 \longrightarrow\mathbb A(\sigma^{j+1}(P),P_{n+r-j-1})
\]
is an unbranched holomorphic covering.
\end{lemma}

\begin{proof}
First suppose $r=1$.
Since $Q\Subset P_n$, Lemma~\ref{new:unmarked-closed} gives an inner
annular neighborhood of $\partial P_{n+1}$ in $P_{n+1}$ that $f$ maps into
$\mathbb A(Q,P_n)$. Let $A\subset P_{n+1}$ be the component of
$f^{-1}(\mathbb A(Q,P_n))$ containing this neighborhood.

For any end $\mathbf e\subset P_{n+1}$ with $\sigma(\mathbf e)\subset Q$,
we have $f(\partial\mathbf e)\subset Q$, whereas
$f(\overline A)\subset\overline{\mathbb A(Q,P_n)}$.
Thus $\overline A\cap\mathbf e=\varnothing$.
By the critical-end condition, this applies to every critical end
in $P_{n+1}$. All critical points in $\Omega$ lie in $W$, so $f$ is
locally univalent on $\overline A$. Hence $A$ has Jordan boundary
curves and $f:A\to\mathbb A(Q,P_n)$ is a finite unbranched covering.
The Riemann--Hurwitz formula shows that $A$ is an annulus with
outer boundary $\partial P_{n+1}$.

Let $P$ be the disk bounded by the inner boundary of $A$. Then
$P\Subset P_{n+1}$, $A=\mathbb A(P,P_{n+1})$, and $f(\partial P)=\partial Q$.
If $Q$ has depth $m$, the argument in the proof of
Lemma~\ref{new:unmarked-closed} gives
$f^{-1}(W_m)=W_{m+1}$ near $\partial P$, and hence
$\partial P\subset\partial W_{m+1}$.
Then the disk $P$ is a piece of depth $m+1$, with $\sigma(P)=Q$.
The piece $P$ is unique: any other piece of depth $m+1$ in $P_{n+1}$
whose image under $\sigma$ is $Q$ would lie in $A$, while $f$ maps
its intersection with $\Omega$ onto $Q\cap\Omega$. This contradicts
$f(A)\subset\mathbb A(Q,P_n)$.

Note that every end
$\mathbf e\subset P_{n+1}$ with $\sigma(\mathbf e)\subset Q$ lies
in $P$, since $\overline A\cap\mathbf e=\varnothing$.
Starting with $Q$, apply the one-step construction successively
along $P_{n+j}\xrightarrow{1}P_{n+j-1}$ for $j=1,\ldots,r$.
For $2\le j\le r$, let $\mathbf c\subset P_{n+j}$ be a critical end.
Since $\sigma^j(\mathbf c)\subset Q$, the preceding containment,
applied through the first $j-1$ pullbacks, shows that
$\sigma(\mathbf c)$ lies in the inner piece obtained at step $j-1$.
Thus the one-step construction applies at step $j$.
\end{proof}

The next lemma follows from Lemmas~\ref{comb:complete-pullback}
and~\ref{ana:real-degree} and the standard modulus formula for
unbranched coverings of annuli; see~\cite{Milnor}.

\begin{lemma}
\label{ana:protected-shell}
Let $P'\xrightarrow{r}Q'$ be a chain, and let
$Q\Subset Q'$ be a piece.
Suppose the critical-end condition of
Lemma~\ref{comb:complete-pullback} holds, and let $P\Subset P'$ be
the unique piece with $\sigma^r(P)=Q$ given by that lemma. Then
\[
 f^r:\mathbb A(P,P')\longrightarrow\mathbb A(Q,Q')
\]
is an unbranched holomorphic covering. If $\nu(P'\to Q')\le L$, then
\begin{equation*}
 \modu(P,P')\ge R^{-L}\modu(Q,Q').
 \end{equation*}
\end{lemma}

\section{Reduction to persistent recurrence}\label{dyn:reduction}

In this section, we prove that every wandering end which is not
persistently recurrent is a point. We begin with limit sets, then
prove the bounded degree criterion.
We keep the absorbing domain $W$ and its
puzzle fixed, and set $b=\#\mathcal E_{\mathrm{crit}}$.
Recall that $R=\deg f$.

To describe where wandering ends accumulate, we use the continuous
map $\sigma:\mathcal E\to\mathcal E$ from
Definition~\ref{def:end-dynamics} and the compact end space of
Lemma~\ref{new:end-realization}.

\begin{definition}\label{def:limit-set}
The \textbf{limit set} of a wandering end $\mathbf e$ is
\[
 \omega(\mathbf e)=\bigcap_{m\ge0}
 \overline{\{\sigma^j(\mathbf e):j\ge m\}}^{\,\mathcal E}.
\]
Here the closure is taken in the end space $\mathcal E$, with
the subspace topology from the quotient sphere $\mathbb S^2$.
Set $\mathcal C(\mathbf e)=\omega(\mathbf e)\cap\mathcal E_{\mathrm{crit}}$.
\end{definition}

\begin{remark}
The set $\omega(\mathbf e)$ is nonempty and compact.
It also satisfies $\sigma(\omega(\mathbf e))=\omega(\mathbf e)$.
Continuity gives the forward inclusion. Conversely, if
$\sigma^{r_j}(\mathbf e)\to\mathbf a$ with $r_j\to\infty$,
compactness gives $\sigma^{r_j-1}(\mathbf e)\to\mathbf b\in\omega(\mathbf e)$
along a subsequence, and $\sigma(\mathbf b)=\mathbf a$.

If $\mathbf a\in\omega(\mathbf e)$ is also wandering, then
$\omega(\mathbf a)\subset\omega(\mathbf e)$; this inclusion
may be strict. Every iterate of $\mathbf e$ is wandering and has
the same limit set as $\mathbf e$.
\end{remark}

For a wandering end $\mathbf e$ and any end $\mathbf a$,
Lemma~\ref{dyn:puzzle}(1)--(2) gives $\mathbf a\in\omega(\mathbf e)$
if and only if, for every $n$, $X_n(\mathbf a)$ contains
$\sigma^k(\mathbf e)$ for infinitely many $k$.

\begin{lemma}\label{red:no-finite-limit}
The limit set of every wandering end is infinite.
\end{lemma}

\begin{proof}
Suppose that a wandering end $\mathbf e$ has a finite limit set.
Since $\sigma(\omega(\mathbf e))=\omega(\mathbf e)$, this limit set
is a finite union of periodic end cycles.

First suppose that $\Omega$ is attracting, and set
\[
 \mathcal P_n=\bigcup_{\mathbf a\in\omega(\mathbf e)}P_n(\mathbf a).
\]
By Lemmas~\ref{new:end-realization} and~\ref{dyn:puzzle}(2),
choose $m\ge1$ so that the pieces of depth $m$ separate the
finitely many ends in $\sigma^{-1}(\omega(\mathbf e))$.

We claim that, for every $n\ge m$ and every end $\mathbf b$,
\[
 \mathbf b\subset\mathcal P_m,\quad
 \sigma(\mathbf b)\subset\mathcal P_n
 \quad\Longrightarrow\quad
 \mathbf b\subset\mathcal P_{n+1}.
\]
Choose $\mathbf a\in\omega(\mathbf e)$ with
$\sigma(\mathbf b)\subset P_n(\mathbf a)$.
By Lemma~\ref{dyn:puzzle}(3), the piece $P_{n+1}(\mathbf b)$
contains a preimage of $\mathbf a$ under $\sigma$.
Since $P_{n+1}(\mathbf b)\subset P_m(\mathbf b)$, the choice of
$m$ shows that this preimage is the end of $\omega(\mathbf e)$
contained in $P_m(\mathbf b)$. Thus $P_{n+1}(\mathbf b)$ is one
of the pieces forming $\mathcal P_{n+1}$.

By compactness of $\mathcal E$, the orbit of $\mathbf e$
eventually stays in $\mathcal P_m$.
Choose $r$ so that $\sigma^{r+j}(\mathbf e)\subset\mathcal P_m$
for all $j\ge0$. For any $j\ge1$, apply the claim backwards
along the orbit from $\sigma^{r+j}(\mathbf e)$ to
$\sigma^r(\mathbf e)$. This gives
$\sigma^r(\mathbf e)\subset\mathcal P_{m+j}$, and hence
\[
 \sigma^r(\mathbf e)\subset\bigcap_{j\ge0}\mathcal P_{m+j}
 =\bigcup_{\mathbf a\in\omega(\mathbf e)}\mathbf a.
\]
Thus $\sigma^r(\mathbf e)$ is a periodic end, contrary to the
wandering property of $\mathbf e$.

In the parabolic case, replace the pieces $P_n(\mathbf a)$ by
the components $X_n(\mathbf a)$. Lemma~\ref{dyn:puzzle}(1)--(3)
gives the same conclusion.
\end{proof}

\subsection{The bounded degree condition}

We state the bounded degree condition in terms of chains and
the number $\nu$ defined in Section~\ref{ana:physical-section};
compare~\cite[Definition~3.1]{ZengFJ}.

\begin{definition}\label{def:bounded-degree}
A wandering end $\mathbf e$ satisfies the \textbf{bounded degree
condition} if there are a fixed piece $Q$ of depth $n$,
integers $r_j\to\infty$, and a constant $L$ such that
\[
 P_{n+r_j}(\mathbf e)\xrightarrow{r_j}Q,
 \qquad \nu(P_{n+r_j}(\mathbf e)\to Q)\le L.
\]
\end{definition}

\noindent\textit{The attracting case.}
Here every end lies in a piece at every depth, and all piece
boundaries lie in $\Omega$.

\begin{lemma}\label{dyn:fixed-attracting}
Suppose that $\Omega$ is attracting. Then every wandering end
satisfying the bounded degree condition is a point.
\end{lemma}

\begin{proof}
Let $Q,r_j$ and $L$ be as in Definition~\ref{def:bounded-degree}.
Choose Jordan disks $V^-\Subset V^+\Subset Q$ with boundaries in
an inner annular neighborhood of $\partial Q$ in $\Omega$
containing no postcritical point of $f$, and with
$\mathbb A(V^-,Q)\subset\Omega$. Thus every end contained
in $Q$ lies in $V^-$. Set
\[
 S_j=\Comp_{\partial\mathbf e}(f^{-r_j}(V^+)),\qquad
 A_j=\Comp_{\partial\mathbf e}(f^{-r_j}(V^-)).
\]
By Lemma~\ref{ana:real-degree}, restricting the proper pullbacks
over $Q$ to $V^+$ gives $\deg(f^{r_j}:S_j\to V^+)\le R^L$.

We first show that, for every fixed depth $k$, we have
$S_j\subset P_k(\mathbf e)$ for all sufficiently large $j$.
The iterates converge uniformly on $\overline W$ to the attracting
fixed point, which lies outside $\overline{V^+}$. Hence
$f^{r_j}(\overline{W_k})\subset f^{r_j-k}(\overline W)$ avoids
$\overline{V^+}$ for large $j$. Since $f^{r_j}(S_j)=V^+$,
the domain $S_j$ is disjoint from $\overline{W_k}$.
It is connected and contains $\partial\mathbf e\subset P_k(\mathbf e)$,
so the claimed inclusion follows.

Since $\partial\mathbf e\Subset A_j\Subset S_j$, we have
\[
 \mathbf e\Subset\Fill(A_j)\Subset\Fill(S_j)\subset P_k(\mathbf e)
\]
for large $j$. The boundaries of $V^-$ and $V^+$ contain no
postcritical point of $f$, so $A_j$ and $S_j$ have finitely many
Jordan boundary curves, and
their fillings are Jordan disks. Lemma~\ref{ana:filled-pullback} gives
\[
 \modu\bigl(\Fill(A_j),\Fill(S_j)\bigr)
 \ge R^{-L}\modu(V^-,V^+)>0.
\]
As $\bigcap_n\overline{P_n(\mathbf e)}=\mathbf e$, we may pass to
a subsequence with $\Fill(S_{j+1})\Subset\Fill(A_j)$.
The resulting annuli are pairwise disjoint and have a uniform
positive lower bound for their moduli. Thus $\mathbf e$ is a
point by~\cite[Theorem~2.16]{McMullen}.
\end{proof}

\noindent\textit{The parabolic case.}
A piece boundary may contain preparabolic vertices, so
the ends $\sigma^{r_j}(\mathbf e)$ in
Definition~\ref{def:bounded-degree} need not all lie in a smaller
Jordan disk compactly contained in $Q$. We will use additional
iterates to reach a depth component disjoint from all forward
images of the piece boundary. Lemma~\ref{red:no-finite-limit}
will allow this choice.

\begin{lemma}\label{dyn:fixed}
Suppose that $\Omega$ is parabolic. Then every wandering end
satisfying the bounded degree condition is a point.
\end{lemma}

\begin{proof}
Let $Q,r_j$ and $L$ be as in Definition~\ref{def:bounded-degree},
and let $n$ be the depth of $Q$. Set
\[
 \mathbf e_j=\sigma^{r_j}(\mathbf e),\qquad
 S_j=\Comp_{\partial\mathbf e}(f^{-r_j}(Q)).
\]
Lemma~\ref{ana:real-degree} gives
$\deg(f^{r_j}:S_j\to Q)\le R^L$.
We will use additional iterates to send $\mathbf e_j$
into fixed Jordan disks $V^-\Subset V^+$ while keeping the
pullback degrees bounded.

We first choose these disks away from all forward images of $\partial Q$.
The set
$\overline{W_n}\setminus\Omega$ consists of finitely many
preparabolic vertices, so $\overline{W_n}$ meets only finitely many ends.
By Lemma~\ref{red:no-finite-limit}, $\omega(\mathbf e)$ is infinite.
Choose $\mathbf b\in\omega(\mathbf e)$ disjoint from $\overline{W_n}$.
By Lemma~\ref{dyn:puzzle}(2), choose $m$ so that
$X_m=X_m(\mathbf b)$ is disjoint from $\overline{W_n}$.

Choose Jordan disks $X_m\Subset V^-\Subset V^+$ such that
$V^+\setminus X_m\subset W_m\subset\Omega$,
$\overline{V^+}\cap\overline{W_n}=\varnothing$, and
$\overline{V^+}\setminus X_m$ contains no postcritical point of $f$.

We next choose the additional iterates by first visits of $\mathbf e_j$ to $V^+$.
Let $s_j$ be the least positive integer such that
$\sigma^{s_j}(\mathbf e_j)\subset V^+$, and set
\[
 U_j=\Comp_{\partial\mathbf e_j}(f^{-s_j}(V^+)).
\]
This time exists because $\mathbf b\in\omega(\mathbf e_j)=\omega(\mathbf e)$.
Since $V^+\setminus X_m\subset\Omega$, this is also the first positive visit
to $X_m$, and $\sigma^{s_j}(\mathbf e_j)\subset X_m$.

We claim that $\deg(f^{s_j}:U_j\to V^+)\le R^b$.
For each $j$, write
\[
 X_{m+s_j-i}=X_{m+s_j-i}(\sigma^i(\mathbf e_j)),
 \qquad 0\le i<s_j.
\]
For each $0\le i<s_j$, Lemma~\ref{dyn:puzzle}(3) says that
$\sigma$ maps the ends in $X_{m+s_j-i}$ onto the ends in
$X_{m+s_j-i-1}$.
The components $X_{m+s_j},\ldots,X_{m+1}$ are pairwise disjoint.
Indeed, if $X_{m+s_j-i}\cap X_{m+s_j-i'}\ne\varnothing$ for
$0\le i<i'<s_j$, then, since depth components are either
disjoint or nested,
\[
 \sigma^i(\mathbf e_j)\subset X_{m+s_j-i}\subset X_{m+s_j-i'}.
\]
Lemma~\ref{dyn:puzzle}(3) then gives
$\sigma^{i+s_j-i'}(\mathbf e_j)\subset X_m$, contrary to the
choice of $s_j$.

Since $V^+\setminus X_m$ contains no postcritical point of $f$,
its inverse components under $f^{s_j-i}$ in $f^i(U_j)$ are annular
boundary collars.
Removing these collars leaves a connected set: it is the intersection
of the nested connected compact sets obtained by removing smaller
boundary collars.
It avoids $W_{m+s_j-i}$ and contains
$\partial\sigma^i(\mathbf e_j)$, so it lies in $X_{m+s_j-i}$.
A path in $\Omega$ from a point of $f^i(U_j)\cap\Omega$ to
$\infty$ meets the collar adjoining the outer boundary of $f^i(U_j)$.
This collar is a connected subset of $f^{-(s_j-i)}(\Omega)$
and meets $\Omega$. Since $\Omega$ is a component of this inverse
image, the collar lies in $\Omega$.
Since $X_{m+s_j-i}$ is full, we obtain
\[
 \Fill(f^i(U_j))\setminus X_{m+s_j-i}\subset\Omega,
 \qquad 0\le i<s_j.
\]
Hence every end in $\Fill(f^i(U_j))$ is contained in $X_{m+s_j-i}$.
Thus each critical end is contained in at most one of the fillings
\[
 \Fill(U_j),\ \Fill(f(U_j)),\ \ldots,\ \Fill(f^{s_j-1}(U_j)).
\]

We now use the degree argument in the proof of
Lemma~\ref{ana:real-degree}. These fillings are Jordan disks
disjoint from $W$. If $\Fill(f^i(U_j))$ contains no critical end,
it contains neither a critical point nor a pole of $f$.
The map $f$ then sends the outer boundary of $f^i(U_j)$ onto
that of $f^{i+1}(U_j)$. The map
$f:\Fill(f^i(U_j))\to\Fill(f^{i+1}(U_j))$ is therefore proper
and unbranched, hence has degree one by the
Riemann--Hurwitz formula. Thus $f:f^i(U_j)\to f^{i+1}(U_j)$
also has degree one.
Each remaining map $f:f^i(U_j)\to f^{i+1}(U_j)$ has degree at most $R$.
Since $b=\#\mathcal E_{\mathrm{crit}}$, there are at most $b$
such steps, giving $\deg(f^{s_j}:U_j\to V^+)\le R^b$.

The inclusion $f^{s_j}(\partial Q)\subset\overline{W_n}$ and
the choice of $V^+$ give $U_j\cap\partial Q=\varnothing$, hence
$U_j\subset Q$ because $\partial\mathbf e_j\subset U_j\cap Q$.
Thus $\Comp_{\partial\mathbf e}(f^{-r_j}(U_j))\subset S_j$,
and $f^{r_j}$ maps this component properly onto $U_j$ with
degree at most $R^L$. Composing with $f^{s_j}:U_j\to V^+$ gives
a proper pullback over $V^+$ of degree at most $R^{L+b}$.
For each fixed depth $k$, the pullbacks over $V^+$ eventually
avoid $\overline{W_k}$, since
$f^{r_j+s_j}(\overline{W_k})\subset\overline W\subset\overline{W_n}$
for $r_j+s_j\ge k$, whereas $\overline{V^+}$ is disjoint from
$\overline{W_n}$. They are connected and contain $\partial\mathbf e$,
so they lie in $P_k(\mathbf e)$.
Since $\sigma^{r_j+s_j}(\mathbf e)\subset X_m\Subset V^-$,
the modulus argument of Lemma~\ref{dyn:fixed-attracting}, with
degree bound $R^{L+b}$, now shows that $\mathbf e$ is a point.
\end{proof}

Together, Lemmas~\ref{dyn:fixed-attracting} and~\ref{dyn:fixed}
prove the bounded degree criterion in both cases.

\subsection{First entries and accumulation of ends}

We now define first entries and estimate the degrees of the
corresponding pullbacks.

\begin{definition}\label{def:first-entry}
Let $\mathbf e$ be a wandering end and let $P$ be a puzzle piece
of depth $n$. If $\{k\ge1:\sigma^k(\mathbf e)\subset P\}$ is
nonempty, its least element is called the \textbf{first entry time}
$r_{\mathbf e}(P)$. The corresponding \textbf{first-entry piece} is
\[
 L_{\mathbf e}(P)=P_{n+r_{\mathbf e}(P)}(\mathbf e).
\]
When $\mathbf e\subset P$, $r_{\mathbf e}(P)$ is the first positive
return time to $P$.
\end{definition}

\begin{lemma}\label{comb:first-entry}
Suppose $L_{\mathbf e}(P)\xrightarrow{r}P$ is a chain,
where $r=r_{\mathbf e}(P)$. Then the pieces
$\sigma^j(L_{\mathbf e}(P))$, $0\le j<r$, are pairwise disjoint.
Thus each critical end is contained in at most one of these
pieces, and $\nu(L_{\mathbf e}(P)\to P)\le b$.
\end{lemma}

\begin{proof}
Let $n$ be the depth of $P$, and write the chain as
$P_{n+r},\ldots,P_n=P$.
If $P_{n+i}\cap P_{n+j}\ne\varnothing$ for $1\le i<j\le r$,
then nesting gives $P_{n+j}\subset P_{n+i}$.
By Lemma~\ref{new:unmarked-closed} applied along the chain,
$\sigma^{r-j}(\mathbf e)\subset P_{n+j}\subset P_{n+i}$, and hence
$\sigma^{r-j+i}(\mathbf e)\subset\sigma^i(P_{n+i})=P$.
This contradicts the minimality of $r$, since $0<r-j+i<r$.
\end{proof}

\begin{lemma}\label{dyn:periodic-case}
Let $\mathbf e$ be a wandering end. If $\omega(\mathbf e)$ contains
a preperiodic end, or a wandering end whose limit set is
properly contained in $\omega(\mathbf e)$, then $\mathbf e$
satisfies the bounded degree condition and is a point.
\end{lemma}

\begin{proof}
We choose $\mathbf a,\mathbf b\in\omega(\mathbf e)$ so that
the forward orbit of $\mathbf a$ under $\sigma$ avoids a neighborhood of
$\mathbf b$ in the topology of the end space $\mathcal E$.
To do this, first suppose that $\omega(\mathbf e)$ contains a
preperiodic end $\mathbf a$. Since $\omega(\mathbf e)$ is infinite
by Lemma~\ref{red:no-finite-limit}, choose $\mathbf b$ outside the
forward orbit of $\mathbf a$ and, in the parabolic case, outside the finite
family of positive forward images of critical ends eventually
mapping to $\mathbf e_p$.
By Lemma~\ref{dyn:puzzle}(1)--(2), choose a piece $Q$ of sufficiently
large depth $m$, visited infinitely often by the orbit of
$\mathbf e$, whose closure meets $\mathbf b$ and avoids the
forward orbit of $\mathbf a$, and whose boundary satisfies the
condition of Lemma~\ref{ana:chain-existence} for every $r\ge1$.
In the parabolic case, $\mathbf b$ may meet $\partial Q$.

Now suppose that every end in $\omega(\mathbf e)$ is wandering,
and choose $\mathbf a$ with
$\omega(\mathbf a)\subsetneq\omega(\mathbf e)$.
Since $\sigma(\omega(\mathbf e))=\omega(\mathbf e)$, we can choose
$\mathbf b$ outside both $\omega(\mathbf a)$ and the forward
orbit of $\mathbf a$.
By Lemma~\ref{dyn:puzzle}(2) and the remark after
Lemma~\ref{ana:chain-existence}, a sufficiently deep piece
$Q=P_m(\mathbf b)$ has closure disjoint from the forward orbit
of $\mathbf a$ and satisfies the same boundary condition.
Since $\mathbf b\in\omega(\mathbf e)$, the orbit of $\mathbf e$
visits $Q$ infinitely often.

For each $n\ge m$, choose a piece $P_n$ of depth $n$ whose closure
meets $\mathbf a$ and which is visited infinitely often by the
orbit of $\mathbf e$. This is possible since
$\mathbf a\in\omega(\mathbf e)$ and there are finitely many
pieces of depth $n$. Set
\[
 r_n=r_{\mathbf e}(P_n),\qquad
 \mathbf e_n=\sigma^{r_n}(\mathbf e),\qquad
 s_n=r_{\mathbf e_n}(Q),\qquad
 Y_n=L_{\mathbf e_n}(Q).
\]
Since $\mathbf e_n$ is a forward iterate of $\mathbf e$, its orbit
also visits $Q$ infinitely often, so $s_n$ exists.
By the choice of $Q$,
Lemma~\ref{ana:chain-existence} gives the full chain
\[
 P_{m+r_n+s_n}(\mathbf e)\xrightarrow{r_n+s_n}Q,
\]
whose piece after $r_n$ steps is $Y_n$.

We claim that $Y_n\subsetneq P_n$. These pieces both contain
$\mathbf e_n$ and hence are nested. If $P_n\subset Y_n$, then
$\overline{Y_n}$ meets $\mathbf a$. If $\mathbf a\subset Y_n$,
the end correspondence along the chain gives
$\sigma^{s_n}(\mathbf a)\subset Q$.
Otherwise $\mathbf a$ meets $\partial Y_n$, and the boundary
correspondence gives
$\sigma^{s_n}(\mathbf a)\cap\partial Q\ne\varnothing$.
Both conclusions contradict the choice of $\overline Q$.

We now count critical-end visits in the full chain.
Its first $r_n$ source pieces are pairwise disjoint. Indeed,
if the pieces after $i$ and $i'$ steps met, where $0\le i<i'<r_n$,
nesting and the subchain from the latter piece to $Y_n$ would give
\[
 \sigma^{i+r_n-i'}(\mathbf e)\subset Y_n\subset P_n,
\]
contrary to the choice of $r_n$ as the first entry time into $P_n$.
The last $s_n$ source pieces are pairwise disjoint by
Lemma~\ref{comb:first-entry}. Each part therefore has at most
$b$ critical-end visits, and the full chain has $\nu\le2b$.

Finally, $r_n\to\infty$. Otherwise Lemma~\ref{dyn:puzzle}(2)
would make $\mathbf a$ a forward iterate of $\mathbf e$.
Every such iterate is wandering and has limit set $\omega(\mathbf e)$,
whereas $\mathbf a$ is preperiodic or has a smaller limit set.
Thus the chains give the bounded degree condition.
Lemmas~\ref{dyn:fixed-attracting} and~\ref{dyn:fixed} show that
$\mathbf e$ is a point.
\end{proof}

Recall that
$\mathcal C(\mathbf e)=\omega(\mathbf e)\cap\mathcal E_{\mathrm{crit}}$.

\begin{lemma}\label{dyn:no-critical-case}
Let $\mathbf e$ be a wandering end. If $\mathcal C(\mathbf e)=\varnothing$,
then $\mathbf e$ satisfies the bounded degree condition and is a point.
\end{lemma}

\begin{proof}
Since $\mathcal C(\mathbf e)=\varnothing$, the orbit of $\mathbf e$
under $\sigma$ does not accumulate on any critical end.
By Lemma~\ref{dyn:puzzle}(1)--(2) and finiteness of the critical
ends, choose $N$ so that each $X_N(\mathbf c)$,
$\mathbf c\in\mathcal E_{\mathrm{crit}}$, contains no end of this
orbit other than possibly $\mathbf c$ itself.
Every end in this orbit is wandering and hence lies in a piece
at every depth. Thus every piece of depth at least $N$ containing
an end of this orbit has boundary disjoint from all critical ends.

Choose a fixed piece $Q$ of depth $N$ containing
$\sigma^{r_j}(\mathbf e)$ for integers $r_j\to\infty$.
Lemma~\ref{new:unmarked-closed} gives the chains
$P_{N+r_j}(\mathbf e)\xrightarrow{r_j}Q$.
A critical-end visit along such a chain can occur only when
the corresponding iterate $\sigma^i(\mathbf e)$ is itself critical.
Since the iterates of $\mathbf e$ are pairwise distinct, each
critical end contributes at most one visit, so
$\nu(P_{N+r_j}(\mathbf e)\to Q)\le b$.
Thus $\mathbf e$ satisfies the bounded degree condition, and
Lemmas~\ref{dyn:fixed-attracting} and~\ref{dyn:fixed} show that it is a point.
\end{proof}

\subsection{Successors}\label{dyn:successors}

For the remaining wandering ends, $\mathcal C(\mathbf e)\ne\varnothing$,
the limit set $\omega(\mathbf e)$ contains no preperiodic end,
and $\omega(\mathbf a)=\omega(\mathbf e)$ for every
$\mathbf a\in\omega(\mathbf e)$.

\begin{definition}\label{def:successor}
Let $\mathbf c$ be a wandering critical end and let $P$
be a piece containing $\mathbf c$.
A piece $Q$ containing
$\mathbf c$ is called a \textbf{successor} of $P$ if
$Q\xrightarrow{s}P$ is a chain for some $s\ge1$ and each critical end
is contained in at most two of the pieces
\[
 Q,\sigma(Q),\ldots,\sigma^s(Q)=P.
\]
Both endpoints are counted. In particular, $\mathbf c$ occurs
only at the endpoints.
\end{definition}

For a successor $Q\xrightarrow{s}P$, we have
$\nu(Q\to P)\le2b-1$. The component of
$f^{-s}(P)$ containing $\partial\mathbf c$ therefore maps properly
onto $P$ under $f^s$ with degree at most $R^{2b-1}$, by
Lemma~\ref{ana:real-degree}.

\begin{lemma}\label{dyn:infinite-successors}
Let $\mathbf e$ be a wandering end and let
$\mathbf c\in\mathcal C(\mathbf e)$ be wandering. Suppose that
some piece containing $\mathbf c$ has infinitely many successors
containing $\mathbf c$.
Then $\mathbf e$ satisfies the bounded degree condition and is a point.
\end{lemma}

\begin{proof}
Let $P$ be such a piece of depth $n$, and write its distinct
successor chains as $Q_{n+s_j}\xrightarrow{s_j}P$.
Since there are finitely many pieces at each depth, we may assume
$s_j\to\infty$. In the parabolic case, choose $N$ so that
$\overline{P_N(\mathbf c)}$ avoids the finite family of forward
images described after Lemma~\ref{ana:chain-existence}, and take
$n+s_j\ge N$. This is possible by Lemma~\ref{dyn:puzzle}(2),
since $\mathbf c$ is wandering and every end in that family
is preperiodic.

The first entry of $\mathbf e$ into $Q_{n+s_j}$ gives a chain
by Lemma~\ref{ana:chain-existence}. Follow this by the successor
chain $Q_{n+s_j}\xrightarrow{s_j}P$.
By Lemma~\ref{comb:first-entry} and the successor condition,
the two parts have at most $b$ and $2b-1$ critical-end visits,
respectively. Thus the full chain has $\nu\le3b-1$.

The lengths of these chains are at least $s_j$ and tend to infinity.
Since their target piece $P$ is fixed, they give the bounded
degree condition for $\mathbf e$. Lemmas~\ref{dyn:fixed-attracting}
and~\ref{dyn:fixed} prove the lemma.
\end{proof}

\subsection{Persistent recurrence}

\begin{definition}\label{def:persistent}
A wandering end $\mathbf e$ is \textbf{persistently recurrent} if
\begin{enumerate}
\item $\mathcal C(\mathbf e)\ne\varnothing$;
\item every end $\mathbf a\in\omega(\mathbf e)$ is wandering and
satisfies $\omega(\mathbf a)=\omega(\mathbf e)$;
\item for every $\mathbf c\in\mathcal C(\mathbf e)$, each piece
containing $\mathbf c$ has only finitely many successors
containing $\mathbf c$.
\end{enumerate}
\end{definition}

\begin{remark}
The finite-successor condition follows the formulation in
\cite[Lemma~C.1]{PQRTY}, which is equivalent, under their recurrence assumptions,
to the finite-children condition in \cite{KSS,QY}.
Following \cite{ZengFJ}, we use the term for wandering ends
satisfying the conditions above. The end itself need not be
critical or recurrent.
\end{remark}

\begin{corollary}\label{dyn:nonpersistent}
Every wandering end which is not persistently recurrent is a point.
\end{corollary}

\begin{proof}
Apply, in order, Lemmas~\ref{dyn:periodic-case},
\ref{dyn:no-critical-case}, and~\ref{dyn:infinite-successors}.
\end{proof}

Section~\ref{comb:section} treats the remaining ends by an
enhanced nest, using the covering estimate of
Section~\ref{ana:kl-section}.

\section{A covering estimate for planar pullbacks}
\label{ana:kl-section}

The previous section reduced the proof to persistently recurrent
ends. The proper pullback domains associated with these ends may
have holes. The proper map from a pullback domain need not extend
to a proper map from the filled domain onto the same target disk.
We first estimate extremal width on the actual pullback, then
obtain a modulus estimate for an annulus in the original plane.
The main tool is the Quasi-Additivity Law
of Kahn and Lyubich~\cite{KL}.

A \textbf{bordered surface} is the interior of a compact connected
Riemann surface with nonempty boundary consisting of finitely many
smooth curves. An \textbf{open archipelago} in a bordered surface $\Sigma$
is a nonempty finite union of bordered subsurfaces
compactly contained in $\Sigma$ and having pairwise disjoint
closures. These subsurfaces need not be planar. A \textbf{closed
archipelago} is a nonempty compact set $A\subset\Sigma$ such that
$\Sigma\setminus A$ is a surface of finite type~\cite[\S2.10.1]{KL}.

For a Riemann surface $\Sigma$ of finite type, possibly disconnected,
write $\partial\Sigma$ for its conformal boundary; for a bordered
surface, this is the boundary specified above.
For a nonempty compact set $A\subset\Sigma$,
let $\Gamma(\Sigma,A)$ consist of
locally rectifiable curves whose interiors lie in
$\Sigma\setminus A$ and whose endpoints lie on
$\partial A$ and $\partial\Sigma$, respectively.
Define the \textbf{extremal width} by
\[
 \mathcal{W}(\Sigma,A)=\inf_\rho\left\{\int_{\Sigma\setminus A}
 \rho^2\,d\mathrm{area}:
 \int_\gamma\rho\,ds\ge1
 \text{ for every }\gamma\in\Gamma(\Sigma,A)\right\},
\]
where $\rho$ ranges over nonnegative Borel functions on
$\Sigma\setminus A$~\cite[\S\S4.1 and 4.4]{KL}.
For a nonempty open set $A\Subset\Sigma$, we write
$\mathcal{W}(\Sigma,A)=\mathcal{W}(\Sigma,\overline A)$.
The extremal width $\mathcal{W}(\Sigma,A)$ is conformally invariant~\cite[\S4-1]{Ahlfors}.

\begin{theorem}[{\cite[p.~562 and \S2.10.1]{KL}}]
\label{ana:qa}
Let $\Sigma$ be a bordered surface.
For pairwise disjoint closed archipelagos $A_1,\ldots,A_l$ in $\Sigma$,
set
\begin{align*}
 X&=\mathcal{W}\left(\Sigma,\bigcup_iA_i\right),&
 Y&=\sum_i\mathcal{W}(\Sigma,A_i),\\
 Z&=\sum_i\mathcal{W}\left(\Sigma\setminus
                   \bigcup_{j\ne i}A_j,A_i\right),&
 \operatorname{Top}&=\operatorname{Top}_{\Sigma}\{A_i\}.
\end{align*}
Here $\operatorname{Top}_{\Sigma}\{A_i\}$ denotes the topological
complexity of the family, as defined in~\cite[\S2.10.1]{KL}.
Then there is a threshold $K$, depending only on $\operatorname{Top}$, such
that $Y\ge K$ implies $Y^2\le2XZ$.
\end{theorem}

A finite branched covering admits a Galois closure~\cite[Proposition~3.1]{KL}.
The following statement also records the boundary behavior of
the Galois closure.

\begin{lemma}
\label{ana:regular-closure}
Let $G:S\to T$ be a degree-$D$ holomorphic branched covering of
bordered surfaces, extending as an unbranched covering on their
boundaries. Then there is a bordered surface $\Sigma$ and branched
coverings
\[
  q:\Sigma\to S,\qquad \pi=G\circ q:\Sigma\to T,
  \qquad n:=\deg\pi\le D!,
\]
such that $\pi$ is Galois: its deck group acts transitively on
each fiber over a regular value. The critical values of $\pi$ are
among those of $G$. The maps $q$ and $\pi$ extend as unbranched
coverings on the boundaries.
\end{lemma}

For the last assertion, choose a collar of $\partial T$ containing
no critical value of $G$. The map $\pi$ is unbranched over this
collar in $T$, and $q$ is unbranched over the preimage of the
collar under $G$ in $S$. Completing these annular coverings gives
the boundary maps of $\pi$ and $q$.

\begin{proposition}
\label{ana:product}
For each integer $D\ge1$ there is $\varepsilon(D)>0$ with the
following property.

Let $G:S\to T$ be a proper holomorphic map
of degree $D$, where $S$ is a finitely connected planar domain and
$T$ is a Jordan disk. Let $B\Subset B'\Subset T$ be Jordan disks.
Choose a component $A'$ of $G^{-1}(B')$, and set
\[
 d=\deg(G:A'\longrightarrow B'),\qquad
 A=A'\cap G^{-1}(B).
\]
The closure $\overline A$ is a closed archipelago; the components of $A$
may have holes.
Then
\begin{equation}
 0<\mathcal{W}(S,A)^{-1}<\varepsilon(D)
 \quad\Longrightarrow\quad
 \mathcal{W}(S,A)^2\le2d^2\mathcal{W}(T,B)\mathcal{W}(B',B).
 \label{ana:product-eq}
\end{equation}
\end{proposition}

\begin{proof}
By conformal invariance and the circle domain theorem, we may take
$T=\mathbb D$ and $S$ to be a circular domain.
Properness excludes point boundary components: $G$ would extend
across such a point with modulus one, contrary to the maximum
principle. Thus $S$ has finitely many analytic boundary curves.
Properness gives $|G(z)|\to1$ as $z\to\partial S$.
Near each boundary curve, $\log|G|$ is harmonic with zero boundary
values. Reflection gives a holomorphic extension of $G$ across
$\partial S$. This extension has no critical point on $\partial S$:
otherwise, its local form would map points of $S$ to both sides
of $\partial\mathbb D$. In particular, $G$ has finitely many
critical values.
Applying Lemma~\ref{ana:regular-closure} gives
\[
 q:\Sigma\longrightarrow S,\qquad
 \pi=G\circ q:\Sigma\longrightarrow T,
 \qquad n=\deg\pi\le D!.
\]
Let $\Gamma$ be the deck group of $\pi$; it acts transitively on
each regular fiber.

Let $U_1,\ldots,U_l$ be the components of $\pi^{-1}(B')$.
Choose a regular value of $\pi$ in $B'$. Then every $U_i$ meets the
$\pi$-fiber over this value. Since $\Gamma$ acts transitively on
this fiber, it permutes the
$U_i$ transitively. Choose $U_1$ with $q(U_1)\subset A'$.
The proper map $q:U_1\to A'$ has image both open and closed in $A'$,
so $q(U_1)=A'$. Set
\[
 m=\deg(q:U_1\longrightarrow A').
\]
Each restriction $\pi:U_i\longrightarrow B'$ has degree $dm$, so
\[
 ldm=n.
\]
In each $U_i$, set
\[
 A_i=U_i\cap\pi^{-1}(\overline B),\qquad 1\le i\le l.
\]
Since $\overline B\Subset B'$, the sets $A_i$ are compactly
contained in the disjoint domains $U_i$. They are closed
archipelagos, their union is $\pi^{-1}(\overline B)$, and $\Gamma$
permutes them transitively.

Apply Theorem~\ref{ana:qa} to the $A_i$ in $\Sigma$.
The identity $\bigcup_iA_i=\pi^{-1}(\overline B)$ allows us to
apply \cite[Lemma~4.7]{KL} to the degree-$n$ map $\pi$, giving
the formula for $X$ below. For $Y$, the restriction $q:U_1\to A'$
has degree $m$ and satisfies
$A_1=(q|_{U_1})^{-1}(\overline A)$. Thus
\cite[Lemma~4.6 and the following remark]{KL}, applied to $q$,
gives $\mathcal{W}(\Sigma,A_1)\ge m\mathcal{W}(S,A)$.
Deck transformations preserve extremal width and permute the
$A_i$ transitively, so all $l$ terms in $Y$ are equal. Therefore
\[
 X=n\mathcal{W}(T,B),\qquad
 Y=l \mathcal{W}(\Sigma,A_1)
   \ge lm \mathcal{W}(S,A)=\frac{n}{d}\mathcal{W}(S,A).
\]
For each $i$, set
$\Sigma_i=\Sigma\setminus\bigcup_{j\ne i}A_j$.
Since all other $A_j$ lie outside $U_i$, every curve defining
$\mathcal{W}(\Sigma_i,A_i)$ must leave $U_i$. The segment up to its first exit
joins $\partial A_i$ to $\partial U_i$ through $U_i\setminus A_i$.
Thus the curves defining $\mathcal{W}(\Sigma_i,A_i)$ overflow those defining
$\mathcal{W}(U_i,A_i)$. As in the proof of Lemma~\ref{ana:filled-pullback},
monotonicity gives $\mathcal{W}(\Sigma_i,A_i)\le \mathcal{W}(U_i,A_i)$.
Choose an inner annular collar of $\partial B'$ in $B'$ containing
no critical value of $\pi$. Its inverse annuli give the conformal
boundary completions of the $U_i$, and $\pi$ extends as an unbranched
covering on the added boundary curves.
Thus \cite[Lemma~4.7]{KL},
applied to $\pi:U_i\to B'$, gives
\[
 Z\le\sum_i\mathcal{W}(U_i,A_i)
   =\sum_i dm \mathcal{W}(B',B)=n\mathcal{W}(B',B).
\]
We next bound the topological complexity in Theorem~\ref{ana:qa}
in terms of $D$. The restriction of $G$ to each boundary curve of
$S$ is a covering of $\partial T$. The degrees of these coverings
are positive and sum to $D$, so $S$ has at most $D$ boundary curves.
Since $S$ is planar, the Riemann--Hurwitz formula bounds the total
ramification of $G$ by $2D-2$, and hence also the number of its
critical values. The map $\pi$ has no other critical values, and
its total ramification over each value is at most $n-1$.
Another application of Riemann--Hurwitz, together with $n\le D!$,
therefore bounds $-\chi(\Sigma)$ in terms of $D$.

For the remaining part of the complexity estimate, take arbitrarily
small smooth Jordan
neighborhoods of $\overline B$ in $B'$ whose boundaries avoid the
critical values of $\pi$. Their preimages under $\pi$ give arbitrarily small smooth
neighborhoods of $\bigcup_i A_i$ with at most $n$ boundary curves:
the restrictions of $\pi$ to these curves are coverings of the
chosen Jordan curve, with positive degrees summing to $n$.
By the definition of closed-archipelago complexity
\cite[\S2.10.1]{KL}, we obtain
\[
 \operatorname{Top}_{\Sigma}\{A_i\}\le-\chi(\Sigma)+n.
\]
Thus this complexity is bounded by a constant depending only on $D$.

For $D\ge2$, choose $K_D\ge1$, depending only on $D$, so that
Theorem~\ref{ana:qa} applies whenever $Y\ge K_D$, and set
$\varepsilon(D)=K_D^{-1}$. Since $n/d=lm\ge1$, if
$\mathcal{W}(S,A)^{-1}<\varepsilon(D)$ then $Y\ge\frac{n}{d}\mathcal{W}(S,A)>K_D$. Thus
\[
 \frac{n^2}{d^2}\mathcal{W}(S,A)^2\le Y^2\le2XZ
 \le2n^2\mathcal{W}(T,B)\mathcal{W}(B',B).
\]
Canceling $n^2$ proves \eqref{ana:product-eq}. If $D=1$, then
$G$ is conformal and $\mathcal{W}(S,A)=\mathcal{W}(T,B)\le \mathcal{W}(B',B)$ by monotonicity.
Thus \eqref{ana:product-eq} holds without a smallness
restriction.
\end{proof}

The next corollary converts the width estimate into a modulus
estimate for puzzle pullbacks. The source annulus lies between
the filled source domain and the filled closure of a component
of the pullback of the inner target disk.

\begin{corollary}
\label{ana:physical-product}
Let $G:S\to T$ be a proper holomorphic map of degree $D$, where $S$ is a
bounded finitely connected planar domain and $T$ is a Jordan disk. Let
$B\Subset B'\Subset T$ be Jordan disks. Choose a component $A_0$ of
$G^{-1}(B)$ and let $A'$ be the component of $G^{-1}(B')$ containing
$A_0$. Set $\hat S=\Fill(S)$ and $d=\deg(G:A'\to B')$.
If $0<\modu\bigl(\hat S\setminus\Fill(\overline{A_0})\bigr)<\varepsilon(D)$, then
\[
 \modu(B,T)\modu(B,B')
 \le2d^2\modu\bigl(\hat S\setminus\Fill(\overline{A_0})\bigr)^2,
\]
where $\varepsilon(D)$ is from Proposition~\ref{ana:product}.
\end{corollary}

\begin{proof}
Since $G$ is proper, $\overline{A_0}\Subset S$. Thus
$\Fill(\overline{A_0})$ is a full continuum compactly contained
in the simply connected domain $\hat S$, and
$\hat S\setminus\Fill(\overline{A_0})$ is an annulus.

Set $A=A'\cap G^{-1}(B)$. Filling the bounded components of
$\mathbb C\setminus\overline{A_0}$ does not change the width from
$\overline{A_0}$ to $\partial\hat S$.
Since $A_0\subset A$ and $S\subset\hat S$, monotonicity gives
\[
 0<\mathcal{W}(S,A)^{-1}\le \mathcal{W}(S,A_0)^{-1}
 \le\modu\bigl(\hat S\setminus\Fill(\overline{A_0})\bigr)<\varepsilon(D).
\]
For the nested Jordan disks, $\mathcal{W}(T,B)=\modu(B,T)^{-1}$ and
$\mathcal{W}(B',B)=\modu(B,B')^{-1}$~\cite[\S4-2]{Ahlfors}.
Proposition~\ref{ana:product} therefore gives
\[
 \modu(B,T)\modu(B,B')\le\frac{2d^2}{\mathcal{W}(S,A)^2}
 \le2d^2\modu\bigl(\hat S\setminus\Fill(\overline{A_0})\bigr)^2.\qedhere
\]
\end{proof}

\begin{remark}
The notation $\modu(B,Q)$ requires the inner domain $B$ to be a
Jordan disk. Here $\Fill(A_0)$ need not be a Jordan disk, so we
use the modulus of $\hat S\setminus\Fill(\overline{A_0})$
directly. When $\Fill(A_0)$ is a Jordan disk, this modulus equals
$\modu(\Fill(A_0),\hat S)$.

The threshold depends on the total degree $D=\deg(G:S\to T)$,
while the factor $d^2$ uses the restricted degree
$d=\deg(G:A'\to B')$.
\end{remark}

\section{Persistently recurrent ends}\label{comb:section}

In this section, we prove that every persistently recurrent end is a point.
We use an enhanced nest about a critical end in the
limit set of the given end. We apply Corollary~\ref{ana:physical-product}
to prove complex bounds for this nest. We then pull the nest
annuli back along first entries of the given end into the nest
pieces to show that this end is a point.

Fix a persistently recurrent end $\mathbf a$, and set
\[
 \mathcal C=\mathcal C(\mathbf a)=\mathcal E_{\mathrm{crit}}\cap\omega(\mathbf a),\qquad
 b_*=\#\mathcal C.
\]
By Definition~\ref{def:persistent}, $b_*\ge1$ and every end
$\mathbf e\in\omega(\mathbf a)$ is wandering with
$\omega(\mathbf e)=\omega(\mathbf a)$. These ends avoid all puzzle
boundaries, and each visits every piece containing an end of
$\omega(\mathbf a)$ infinitely often. Write
$\Post(\mathcal C)=\bigcup_{\mathbf c\in\mathcal C}\bigcup_{j\ge1}
\sigma^j(\mathbf c)$, a union of ends in $\Chat$.
Recall the chain notation $P\xrightarrow{r}Q$ from
Definition~\ref{def:puzzle-chain};
$\nu(P\to Q)$ counts its critical-end visits.

\subsection{The initial depth}\label{comb:initial-depth}

The following lemma fixes the initial depth used in the nest construction.
Write $h(P)$ for the depth of $P$ and
$h(P',P)=h(P')-h(P)$.

\begin{lemma}\label{comb:initial-depth-properties}
There is an integer $N\ge1$ such that the following hold for all
pieces and components of depth at least $N$.
\begin{enumerate}
\item Every depth-$n$ component containing an end of
$\omega(\mathbf a)$ contains at most one critical end and none
outside $\mathcal C$. Every piece containing an end of
$\omega(\mathbf a)$ has boundary disjoint from all critical ends.
\item If $Q\xrightarrow{r}P$ is a chain with $Q\subset P$ and
$Q$ contains an end of $\omega(\mathbf a)$, then $Q\Subset P$.
\item Let $I,P,P'$ be pieces with $P'\subsetneq P$.
Suppose that $P'$ contains an end of $\omega(\mathbf a)$ and
\[
 \sigma^{u'}(P')=\sigma^u(P)=I,\qquad u'>u\ge0,
\]
where the positive iterates are defined by chains. Then $P'\Subset P$.
\end{enumerate}
\end{lemma}

\begin{proof}
For (1), Lemma~\ref{dyn:puzzle}(2) gives a depth separating the
finitely many critical ends.
We claim that, for all sufficiently large $n$, $X_n(\mathbf e)$
contains no critical end outside $\mathcal C$ for every
$\mathbf e\in\omega(\mathbf a)$.
Otherwise, finiteness and nesting give a critical end
$\mathbf c\notin\mathcal C$ and ends
$\mathbf e_n\in\omega(\mathbf a)$ with
$\mathbf e_n\subset X_n(\mathbf c)$.
By compactness, a subsequence converges, and
Lemma~\ref{dyn:puzzle}(2) forces its limit to be $\mathbf c$.
Thus $\mathbf c\in\omega(\mathbf a)$, a contradiction.

To prove the boundary assertion, suppose that a piece $P$
contains an end of $\omega(\mathbf a)$ and $\partial P$ meets a
critical end $\mathbf c$. By Lemma~\ref{dyn:puzzle}(1), the whole
end $\mathbf c$ lies in the depth component containing
$\overline P$. Hence $\mathbf c\in\mathcal C$, contradicting the
fact that ends in $\mathcal C$ avoid all piece boundaries.

For (2), fix a depth $n$ satisfying (1).
For each $\mathbf e\in\omega(\mathbf a)$,
Lemma~\ref{dyn:puzzle}(2) gives a deeper piece compactly contained
in $P_n(\mathbf e)$. Since $\omega(\mathbf a)$ is compact and
there are finitely many depth-$n$ pieces, we can choose $m>n$ such that
\[
 P_m(\mathbf e)\Subset P_n(\mathbf e)
 \qquad\bigl(\mathbf e\in\omega(\mathbf a)\bigr).
\]

The lengths of the chains in (2) tend uniformly to infinity as
$h(P)\to\infty$. Otherwise, take chains $Q_j\xrightarrow{r}P_j$
with fixed $r$ and $h(P_j)\to\infty$.
Choose $\mathbf e_j\in\omega(\mathbf a)$ contained in $Q_j$ and,
by compactness, pass to a subsequence converging to
$\mathbf e\in\omega(\mathbf a)$.
Since $\mathbf e_j$ and $\sigma^r(\mathbf e_j)$ lie in $P_j$,
Lemma~\ref{dyn:puzzle}(2) and continuity give
$\sigma^r(\mathbf e)=\mathbf e$,
contrary to Definition~\ref{def:persistent}.

Choose $N\ge m$ so that $r\ge m-n$ for every chain in (2) with
$h(P)\ge N$. For such a chain, set $k=h(P)-n$.
By the choice of $n$ and Lemma~\ref{new:unmarked-closed}, the nested
pieces $\sigma^k(Q)\subset\sigma^k(P)$ contain the same end of
$\omega(\mathbf a)$, and their depth difference is $r=h(Q,P)$.
Since $\sigma^k(P)$ has depth $n$ and $r\ge m-n$, the choice of
$m$ gives $\sigma^k(Q)\Subset\sigma^k(P)$.
The boundaries of these image pieces are $f^k(\partial Q)$ and
$f^k(\partial P)$, respectively. The image boundaries are disjoint,
so $\partial Q\cap\partial P=\varnothing$ and $Q\Subset P$.

For (3), the piece $\sigma^u(P')\subset I$ contains an end of
$\omega(\mathbf a)$ and has a chain to $I$ of length $u'-u>0$.
By (2), $\sigma^u(P')\Subset I$. The same boundary argument under
$f^u$ gives $P'\Subset P$.
\end{proof}

From now on, all pieces considered have depth at least the integer
$N$ given by Lemma~\ref{comb:initial-depth-properties}.

\subsection{The enhanced nest}

We adapt the enhanced-nest construction of \cite{KSS,PQRTY,RY}
to our setting. In the cited constructions, the relevant iterates
give proper holomorphic maps between puzzle pieces. Here a critical
end may correspond to a regular point of $F$, and a chain
$P\xrightarrow{r}Q$ need not give a proper map $f^r:P\to Q$.
The actual proper pullbacks may be multiply connected.
We therefore verify in Appendix~\ref{app:enhanced-nest} that the same
combinatorial construction applies in this setting.
We use the preceding pullback
lemmas to construct the inner pieces and count critical-end visits
to bound the degrees of the actual proper pullbacks.
The following proposition records the properties needed below.

\begin{proposition}\label{comb:enhanced-nest}
Fix $\mathbf c_0\in\mathcal C$ and a piece $K_0$ containing
$\mathbf c_0$, and set $N_0=10(b_*+1)^2$.
There are pieces $K_n,K'_n$ and integers $p_n>p'_n>0$,
for $n\ge1$, with chains
\[
 K_n\xrightarrow{p_n}K_{n-1},\qquad
 K'_n\xrightarrow{p'_n}K_{n-1}.
\]
Set
\[
 s(n,l)=p_n+\cdots+p_{l+1}\quad(0\le l<n),\qquad s(n,n)=0.
\]
Then $K_n\xrightarrow{s(n,l)}K_l$ for $0\le l<n$, and
$h(K_n,K'_n)=p_n-p'_n$. For every $n\ge1$, the following hold.
\begin{enumerate}
\item $\mathbf c_0\subset K_n\Subset K'_n\Subset K_{n-1}$, and
$\mathbb A(K_n,K'_n)\cap\Post(\mathcal C)=\varnothing$.
\item $\nu(K_n\to K_{n-1}),\,\nu(K'_n\to K_{n-1})\le N_0$.
The proper pullbacks of $K_{n-1}$ along these chains have degree
at most $R^{N_0}$ under $f^{p_n}$ and $f^{p'_n}$, respectively.
\item $s(n,0)\le2p_n$ and $h(K_{n+2},K'_{n+2})>s(n,0)$.
\end{enumerate}
\end{proposition}

We fix such a nest for the rest of this section.

\subsection{Pullbacks of the nest}
\label{comb:landing-section}

For $n\ge1$ and each end $\mathbf e\in\omega(\mathbf a)$ contained in $K_n$,
we construct a nested pair inside $K_n$ by pulling back
$(K_n,K'_n)$. We use these pairs to bound moduli below by sums of
nest moduli and to bound the degree $d$ over $B'$ in
Corollary~\ref{ana:physical-product}.
In Proposition~\ref{ana:amplification}, this bound for $d$ is
independent of the number of nest levels crossed, while the bound
for the total degree $D$ may depend on that number.

Since $\sigma^{s(n,0)}(\mathbf e)\in\omega(\mathbf a)$, this end
enters $K_n$. Let
$r_n(\mathbf e)=r_{\sigma^{s(n,0)}(\mathbf e)}(K_n)$ be its
first positive entry time, and set $m=s(n,0)+r_n(\mathbf e)$.
Pull back $K_n$ and $K'_n$ in $m$ steps along the
$\sigma$-orbit of $\mathbf e$. Lemma~\ref{comb:initial-depth-properties}(1)
and Lemma~\ref{new:unmarked-closed} give the chains below;
denote their initial pieces by
\[
 H_n(\mathbf e)\Subset H_n'(\mathbf e),\qquad
 H_n(\mathbf e)\xrightarrow{m}K_n,\qquad
 H_n'(\mathbf e)\xrightarrow{m}K'_n.
\]
Set $L_0=2(N_0+b_*)$ and $C_0=R^{L_0}$, with $N_0$ as in
Proposition~\ref{comb:enhanced-nest}.

\begin{proposition}\label{comb:long-landing}
The pieces just defined have the following properties.
\begin{enumerate}
\item $\mathbf e\subset H_n(\mathbf e)\Subset H_n'(\mathbf e)\Subset K_n$.
\item The components of $f^{-m}(K_n)$ and $f^{-m}(K'_n)$
containing $\partial\mathbf e$ map under $f^m$ properly onto $K_n$
and $K'_n$, respectively, with degree at most $C_0$. The same bound
holds for the proper pullback corresponding to any subchain of either chain.
Moreover,
\[
 f^m:\mathbb A(H_n(\mathbf e), H_n'(\mathbf e))
       \longrightarrow\mathbb A(K_n, K'_n)
\]
is an unbranched covering of degree at most $C_0$.
\item For $n\ge2$,
$\sigma^{p_n}(\mathbf e)\subset \sigma^{p_n}(H_n(\mathbf e))\subset
\sigma^{p_n}(H_n'(\mathbf e))\subset H_{n-1}(\sigma^{p_n}(\mathbf e))$.
\item For $0\le l<n$, set
$B=\sigma^{s(n,l)}(H_n(\mathbf e))$ and $B'=\sigma^{s(n,l)}(H_n'(\mathbf e))$. Then
\[
 \modu(B, B')\ge C_0^{-1}\modu(K_n, K'_n),
 \qquad
 \modu(B, K_l)
    \ge C_0^{-1}\sum_{i=l+1}^n\modu(K_i, K'_i).
\]
\end{enumerate}
\end{proposition}

\begin{proof}
(1) The construction gives
\[
 h(H_n'(\mathbf e))-h(K_n)
 =s(n,0)+r_n(\mathbf e)+p'_n-p_n
 \ge p'_n>0.
\]
The pieces contain $\mathbf e$, and
$h(H_n(\mathbf e),H_n'(\mathbf e))=h(K_n,K'_n)>0$.
Thus $H_n(\mathbf e)\subsetneq H_n'(\mathbf e)\subsetneq K_n$.
All three pieces have chains to $K_0$, so
Lemma~\ref{comb:initial-depth-properties}(3) gives the compact containments in (1).

(2) Decompose the inner chain as
\[
 H_n(\mathbf e)\xrightarrow{s(n,0)}
 L_{\sigma^{s(n,0)}(\mathbf e)}(K_n)
 \xrightarrow{r_n(\mathbf e)}K_n.
\]
The last chain has at most $b_*$ critical-end visits, by
Lemmas~\ref{comb:first-entry} and~\ref{comb:initial-depth-properties}(1).
It remains to bound the number of critical-end visits in the first chain.
Let $u=r_{\sigma^{p_n}(\mathbf e)}(K_n)$ be the first positive
entry time of $\sigma^{p_n}(\mathbf e)$ into $K_n$.
Then $p_n+u\le m$, and nesting gives
$\sigma^{p_n+u}(H_n(\mathbf e))\subset K_n$.
The initial subchain of length $p_n$ is a restriction of
$K_n\xrightarrow{p_n}K_{n-1}$.
Any following steps up to time $\min\{s(n,0),p_n+u\}$ form a
restriction of an initial subchain of the first-entry chain of
$\sigma^{p_n}(\mathbf e)$ into $K_n$.
Any remaining subchain before time $s(n,0)$ has length at most
$p_n$, since $s(n,0)\le2p_n$ by
Proposition~\ref{comb:enhanced-nest}(3), and is a restriction of an
initial subchain of $K_n\xrightarrow{p_n}K_{n-1}$.
By Proposition~\ref{comb:enhanced-nest}(2) and
Lemma~\ref{comb:first-entry}, the first chain has at most
$2N_0+b_*$ critical-end visits. Hence
\[
 \nu(H_n(\mathbf e)\to K_n)\le2(N_0+b_*)=L_0.
\]

At each critical-end visit of the outer chain, the critical end
belongs to $\mathcal C$, by Lemma~\ref{comb:initial-depth-properties}(1),
and its terminal image lies in $\Post(\mathcal C)\cap K'_n\subset K_n$,
by Proposition~\ref{comb:enhanced-nest}(1).
Lemma~\ref{comb:complete-pullback} then gives $H_n(\mathbf e)$ as
the unique inner pullback, and the two chains have the same
critical-end visits. Lemma~\ref{ana:real-degree} bounds the degrees
of both chains and their subchains by $C_0$;
Lemma~\ref{ana:protected-shell} gives the unbranched annular
covering with the same degree bound.

(3) We only need to prove the last inclusion; the others follow
from construction. Set $\mathbf e_{n-1}=\sigma^{p_n}(\mathbf e)$. Since
$\sigma^{s(n,0)}(\mathbf e)=\sigma^{s(n-1,0)}(\mathbf e_{n-1})$,
$r_n(\mathbf e)$ and $r_{n-1}(\mathbf e_{n-1})$ are the first
positive entry times of $\sigma^{s(n,0)}(\mathbf e)$ into $K_n$
and $K_{n-1}$, respectively. As $K_n\subset K_{n-1}$, we have
$r_{n-1}(\mathbf e_{n-1})\le r_n(\mathbf e)$. The pieces
$\sigma^{p_n}(H_n'(\mathbf e))$ and $H_{n-1}(\mathbf e_{n-1})$
contain $\mathbf e_{n-1}$. Since $h(K'_n,K_{n-1})=p'_n$,
the pullback lengths give the depth difference
\[
 p'_n+r_n(\mathbf e)-r_{n-1}(\mathbf e_{n-1})\ge p'_n>0.
\]
Nesting gives the last inclusion.

(4) For $l<i\le n$, set
\[
 \mathbf e_i=\sigma^{s(n,i)}(\mathbf e),\qquad
 B_i=\sigma^{s(i,l)}(H_i(\mathbf e_i)),\qquad
 B'_i=\sigma^{s(i,l)}(H_i'(\mathbf e_i)).
\]
Then $B_n=B$, $B'_n=B'$, and (3) gives $B'_i\subset B_{i-1}$
for $l+2\le i\le n$. By (1) and nesting, $B'_i\subsetneq K_l$.
Using $K_0$ as the common target in
Lemma~\ref{comb:initial-depth-properties}(3) gives
$B_i\Subset B'_i\Subset K_l$.
The corresponding annuli $\mathbb A(B_i, B'_i)$
are disjoint and surround $B$ in $K_l$.
By (2) and Lemma~\ref{ana:protected-shell}, the map
\[
 f^{s(l,0)+r_i(\mathbf e_i)}:
 \mathbb A(B_i,B'_i)\longrightarrow\mathbb A(K_i,K'_i)
\]
is an unbranched covering of degree at most $C_0$. Consequently,
\[
 \modu(B_i, B'_i)
      \ge C_0^{-1}\modu(K_i, K'_i).
\]
Taking $i=n$ proves the first bound. Adding these inequalities
and applying Gr\"otzsch's inequality proves the second.
\end{proof}

\subsection{Complex bounds}\label{ana:amplification-section}

We prove complex bounds by contradiction. Suppose that the nest
moduli have a subsequence tending to zero. We first use
Lemma~\ref{ana:filled-pullback} to obtain the small source modulus
required by Corollary~\ref{ana:physical-product}. For a suitable
choice of this subsequence, Proposition~\ref{comb:long-landing}(4)
gives lower bounds for the moduli of the two target annuli.
When the pullback crosses sufficiently many nest levels, these
bounds contradict the upper bound for their product in
Corollary~\ref{ana:physical-product}.

\begin{proposition}\label{ana:amplification}
Let $(K_n,K'_n)$ be the nest of Proposition~\ref{comb:enhanced-nest}.
Then
\[
 \liminf_{n\to\infty}\modu(K_n, K'_n)>0.
\]
In particular, $\bigcap_n\overline{K_n}$ is a point.
\end{proposition}

\begin{proof}
Keep the constants $N_0$, $L_0=2(N_0+b_*)$, and $C_0=R^{L_0}$
from Proposition~\ref{comb:enhanced-nest} and
Proposition~\ref{comb:long-landing}, and set $C_2=R^{2N_0}$.
Choose an integer $Z>2C_0^4C_2^2$ and then set
\[
 D_Z=R^{ZN_0},\qquad
 \varepsilon_Z=\min_{1\le D\le D_Z}\varepsilon(D)>0,
\]
where $\varepsilon(D)$ is from Corollary~\ref{ana:physical-product}.
These constants are independent of $n$.

Suppose that $\liminf_n\modu(K_n,K'_n)=0$.
Choose a subsequence $n_j\to\infty$ such that
\[
 \lim_{j\to\infty}\modu(K_{n_j}, K'_{n_j})=0,
 \qquad
 \modu(K_{n_j}, K'_{n_j})
       \le\modu(K_i, K'_i)\quad(1\le i\le n_j).
\]
Take $n=n_j-2$, with $j$ large enough that
\[
 n>Z,\qquad
 C_2\modu(K_{n+2}, K'_{n+2})<\varepsilon_Z.
\]
Set
\[
 \mathbf e_n=\sigma^{p'_{n+2}+p_{n+1}}(\mathbf c_0),\qquad
 J=\sigma^{p'_{n+2}+p_{n+1}}(K_{n+2}).
\]
The outer chain
$K'_{n+2}\xrightarrow{p'_{n+2}+p_{n+1}}K_n$ determines a
proper pullback map of degree at most $C_2$.
The piece $J\subset K_n$ contains $\mathbf e_n$ and has a chain to
$K_n$ of length $p_{n+2}-p'_{n+2}>0$. Thus $J\Subset K_n$
by Lemma~\ref{comb:initial-depth-properties}(2).
The inner chain is
$K_{n+2}\xrightarrow{p'_{n+2}+p_{n+1}}J$.
By Lemma~\ref{ana:exact-hull}, the corresponding inverse components
containing $\partial\mathbf c_0$ have fillings $K'_{n+2}$ and
$K_{n+2}$, respectively.
Lemma~\ref{ana:filled-pullback} therefore gives
\[
 \modu(J, K_n)
       \le C_2\modu(K_{n+2}, K'_{n+2}).
\]

Proposition~\ref{comb:enhanced-nest}(3) gives
\[
 h(J,K_n)=h(K_{n+2},K'_{n+2})
 >s(n,0).
\]
The chain $J\xrightarrow{h(J,K_n)}K_n$ therefore gives a
positive entry of $\sigma^{s(n,0)}(\mathbf e_n)$ into $K_n$
after at most $h(J,K_n)-s(n,0)$ further iterates. Thus
$r_n(\mathbf e_n)\le h(J,K_n)-s(n,0)$, and comparison of depths gives
$J\subset H_n(\mathbf e_n)$. Hence
\begin{equation}\label{ana:source-compression}
 \modu(H_n(\mathbf e_n),K_n)
 \le\modu(J,K_n)
 \le C_2\modu(K_{n+2},K'_{n+2}).
\end{equation}

For the long pullback, set
\[
 B=\sigma^{s(n,n-Z)}(H_n(\mathbf e_n)),\qquad
 B'=\sigma^{s(n,n-Z)}(H_n'(\mathbf e_n))
\]
and let $S$, $A_0$, and $A'$ be the components
\[
 \begin{aligned}
 S&=\Comp_{\partial\mathbf e_n}(f^{-s(n,n-Z)}(K_{n-Z})),\\
 A_0&=\Comp_{\partial\mathbf e_n}(f^{-s(n,n-Z)}(B)),\qquad
 A'=\Comp_{\partial\mathbf e_n}(f^{-s(n,n-Z)}(B')).
 \end{aligned}
\]
Then $A_0\subset A'\subset S$. Lemma~\ref{ana:exact-hull} shows
that $S$ is finitely connected and gives
\[
 \Fill(S)=K_n,\qquad
 \Fill(A_0)=H_n(\mathbf e_n),\qquad
 \Fill(A')=H_n'(\mathbf e_n).
\]
Since $H_n(\mathbf e_n)$ is a Jordan disk,
$\Fill(\overline{A_0})=\overline{H_n(\mathbf e_n)}$.
Proposition~\ref{comb:enhanced-nest}(2) and
Proposition~\ref{comb:long-landing}(2) give, respectively,
\[
 D:=\deg(f^{s(n,n-Z)}:S\to K_{n-Z})\le D_Z,\qquad
 d:=\deg(f^{s(n,n-Z)}:A'\to B')\le C_0.
\]
By Proposition~\ref{comb:long-landing}(4) and the choice of $n_j$,
\[
 \begin{aligned}
 \modu(B,B')&\ge C_0^{-1}\modu(K_{n+2},K'_{n+2}),\\
  \modu(B,K_{n-Z})&\ge C_0^{-1}\sum_{i=n-Z+1}^n\modu(K_i,K'_i)
 \ge ZC_0^{-1}\modu(K_{n+2},K'_{n+2}).
 \end{aligned}
\]
Apply Corollary~\ref{ana:physical-product} to
$f^{s(n,n-Z)}:S\to K_{n-Z}$,
with the components $A_0\subset A'$. By
\eqref{ana:source-compression}, we have
$0<\modu(H_n(\mathbf e_n),K_n)<\varepsilon_Z\le\varepsilon(D)$.
Thus
\[
  \modu(B,K_{n-Z})\modu(B,B')
 \le2d^2\modu(H_n(\mathbf e_n),K_n)^2
 \le2C_0^2C_2^2\modu(K_{n+2},K'_{n+2})^2.
\]
The two lower bounds and this upper bound give, after division
by $\modu(K_{n+2},K'_{n+2})^2>0$,
$Z\le2C_0^4C_2^2$, contrary to our choice of $Z$.
This proves the complex bounds.

Finally, the annuli $\mathbb A(K_n,K'_n)$ are disjoint and
surround $\bigcap_n\overline{K_n}$. Their moduli have divergent
sum. By~\cite[Theorem~2.16]{McMullen}, the intersection is a point.
\end{proof}

\subsection{Persistently recurrent ends and the main theorem}
\label{dyn:annuli-section}

\begin{proposition}\label{dyn:persistent-singleton}
Every persistently recurrent end of $f$ is a point.
\end{proposition}

\begin{proof}
By Lemma~\ref{dyn:puzzle}(2), choose a depth that separates each
critical end outside $\mathcal C$ from every distinct end in the
forward orbit of $\mathbf a$ under $\sigma$. Such a common depth
exists because there are finitely many critical ends and none
outside $\mathcal C$ belongs to $\omega(\mathbf a)$.
Both the orbit of $\mathbf a$ and
$\mathcal C$ consist of wandering ends. Hence sufficiently deep
pieces along this orbit have boundaries disjoint from all critical ends.

Since $\mathbf c_0\in\omega(\mathbf a)$, let
$r_n=r_{\mathbf a}(K_n)$ and set $Q_n=L_{\mathbf a}(K_n)$.
For all sufficiently large $n$, pull $K'_n$ back along the same
orbit segment to obtain $Q'_n$; the chains exist by
Lemma~\ref{new:unmarked-closed}.
Consider a critical end $\mathbf c$ encountered at step $i$ of
the outer chain. If $\mathbf c\in\mathcal C$, then
$\sigma^{r_n-i}(\mathbf c)\subset\Post(\mathcal C)\cap K'_n\subset K_n$.
Otherwise, $\mathbf c=\sigma^i(\mathbf a)$, so its terminal image
also lies in $K_n$. Thus Lemma~\ref{comb:complete-pullback} gives
$Q_n\Subset Q'_n$, and the two chains have the same critical-end visits.
By Lemmas~\ref{comb:first-entry} and~\ref{ana:protected-shell},
\[
 \modu(Q_n,Q'_n)\ge R^{-b}\modu(K_n,K'_n).
\]
Proposition~\ref{ana:amplification} therefore gives a uniform
positive lower bound for these moduli.

Since $h(Q'_n)\to\infty$, Lemma~\ref{dyn:puzzle}(2) gives a
subsequence with each outer piece compactly contained in the
preceding inner piece. The corresponding annuli surround
$\mathbf a$ and have divergent sum of moduli.
Gr\"otzsch's inequality implies that $\mathbf a$ is a point.
\end{proof}

\begin{proof}[Proof of Theorem~\ref{thm:main}]
Corollary~\ref{dyn:nonpersistent} and
Proposition~\ref{dyn:persistent-singleton} show that every wandering
end of $\Omega$ is a point. Since every wandering boundary component
of $\Omega$ is the boundary of a wandering end, it is a point.
Lemma~\ref{dyn:basin-reduction} completes the proof.
\end{proof}

\noindent\textbf{Acknowledgments.}
The author thanks Guizhen Cui, Wenjuan Peng, and Yongcheng Yin
for helpful discussions and suggestions. The author also thanks
ChatGPT for suggesting a key idea in the proof of
Proposition~\ref{ana:product} and for helping to improve the language
of this paper.
The author takes full responsibility for the
accuracy and integrity of the paper, including its results, proofs,
and references. This work was partially
supported by the National Natural Science Foundation of China
(Grant No.~12271115).

\appendix
\section{Construction of the enhanced nest}\label{app:enhanced-nest}

We prove Proposition~\ref{comb:enhanced-nest}.
We retain the notation $\mathbf a,\mathcal C,b_*$ of
Section~\ref{comb:section} and consider only pieces of depth at least
$N$, as in Lemma~\ref{comb:initial-depth-properties}.

A chain $P\xrightarrow{r}Q$ need not give a proper map
$f^r:P\to Q$. The degree bounds therefore concern the actual
proper pullbacks supplied by Lemma~\ref{ana:exact-hull}.
Lemma~\ref{ana:real-degree} bounds these degrees by counting
critical-end visits, including visits to ends corresponding
to regular points of $F$.
We use Lemma~\ref{comb:first-entry} to bound visits along
first-entry chains and Lemma~\ref{comb:complete-pullback}
to construct the inner puzzle pieces.
Lemma~\ref{comb:initial-depth-properties} gives the compact
containments and, together with Lemma~\ref{new:unmarked-closed},
the chains used below.
We first construct the nest operators, then use successors to
construct the nest and prove its properties.

\subsection{The nest operators}

The following construction of $\mathcal A$ and $\mathcal B$
follows \cite[proof of Lemma~8.2]{KSS} and
\cite[Lemma~4.4 and Appendix~C]{PQRTY}.

\begin{proposition}\label{comb:AB}
Fix $\mathbf c_0\in\mathcal C$ and a piece $I$ of the chosen puzzle
containing $\mathbf c_0$. Then there are pieces
$\mathcal A(I)\Subset\mathcal B(I)\Subset I$, both containing
$\mathbf c_0$, and positive integers
$\alpha=\alpha(I)>\beta=\beta(I)>0$ such that
\[
 \mathcal A(I)\xrightarrow{\alpha}I,\qquad
 \mathcal B(I)\xrightarrow{\beta}I,
 \qquad
 \mathbb A(\mathcal A(I),\mathcal B(I))\cap\Post(\mathcal C)
 =\varnothing.
\]
Moreover, $\nu(\mathcal A(I)\to I),\,
\nu(\mathcal B(I)\to I)\le3(b_*+1)^2$.
\end{proposition}

\begin{proof}[Sketch of the proof]
\textit{Step 1. Choose the auxiliary pairs.}
As in \cite[Proposition~B.1]{PQRTY}, choose a family of pairs
\[
 \mathbf c\subset P'_{\mathbf c}\Subset P_{\mathbf c}
 \qquad(\mathbf c\in\mathcal C).
\]
Both $P_{\mathbf c}$ and $P'_{\mathbf c}$ have chains to $I$
of positive length, with the chain from $P'_{\mathbf c}$ being longer. Moreover,
\[
 \nu(P_{\mathbf c}\to I)\le2b_*^2,\qquad
 \nu(P'_{\mathbf c}\to I)\le2b_*^2+b_*.
\]
For every end $\mathbf e\subset\Post(\mathcal C)\cap
\mathbb A(P'_{\mathbf c},P_{\mathbf c})$, there are a piece $U$,
an end $\mathbf c'\in\mathcal C$, and an integer $r\ge1$ such that
\[
 \mathbf e\subset U\subset\mathbb A(P'_{\mathbf c},P_{\mathbf c}),
 \qquad U\xrightarrow{r}P_{\mathbf c'},
 \qquad \nu(U\to P_{\mathbf c'})=0.
\]
To ensure positive chain lengths, apply the cited selection to
$L_{\mathbf c_0}(I)$ and append the first-return chain
$L_{\mathbf c_0}(I)\to I$. This adds at most $b_*$ critical-end
visits and gives the bounds above.

\noindent
\textit{Step 2. Choose $Q$ and $Q'$.}
Following \cite[Appendix~C, proof of Lemma~4.4]{PQRTY}, consider
all chains
\[
 Q\xrightarrow{t}P_{\mathbf c},\qquad
 \mathbf c'\subset Q,\qquad
 \nu(Q\to P_{\mathbf c})=1,
\]
where $t\ge1$ and $\mathbf c,\mathbf c'\in\mathcal C$.
Starting at the last critical-end visit of a first-return chain
gives a chain of this form. Conversely, prepending the first entry
of $\mathbf c$ into $Q$, if needed, gives a successor chain to
$P_{\mathbf c}$. The finite-successor condition therefore gives
only finitely many choices.
Choose and fix a chain $Q\xrightarrow{t}P_{\mathbf c}$ of greatest length $t$.
The corresponding ends $\mathbf c$ and $\mathbf c'$ remain fixed
throughout the rest of this proof.
By maximality, $\sigma^t(\mathbf c')\subset P'_{\mathbf c}$.
Lemma~\ref{comb:complete-pullback} therefore gives a piece $Q'$ with
\[
 \mathbf c'\subset Q'\Subset Q,\qquad
 Q'\xrightarrow{t}P'_{\mathbf c}.
\]
The maximality of $t$ also gives
\[
 \Post(\mathcal C)\cap Q\subset Q'.
\]
Indeed, for an end in $\Post(\mathcal C)\cap\mathbb A(Q',Q)$,
apply Step~1 to its $\sigma^t$-image. Pull the resulting chain
back along a critical orbit through that end and start at the
last critical-end visit. This gives a
chain of the displayed form with length greater than $t$.

Let $u,u'$ be the lengths of the chains from $Q,Q'$ to $I$,
respectively. Then $1\le u<u'$ and
\[
 Q\xrightarrow{u}I,\qquad Q'\xrightarrow{u'}I,\qquad
 \nu(Q\to I)\le2b_*^2+1.
\]

\noindent
\textit{Step 3. Construct $\mathcal B(I)$ and $\mathcal A(I)$ and count critical-end visits.}
If $\mathbf c_0\subset Q$, set $\mathcal B(I)=Q$.
Otherwise, set $\mathcal B(I)=L_{\mathbf c_0}(Q)$.
This entry exists because $Q$ contains
an end of $\mathcal C\subset\omega(\mathbf c_0)$.
Concatenate the entry chain, when present, with
$Q\xrightarrow{u}I$, and denote the total length by $\beta$. Thus
\[
 \mathcal B(I)\xrightarrow{\beta}I.
\]
Set $\mathbf e=\sigma^\beta(\mathbf c_0)$.
Pull the first-return piece $L_{\mathbf e}(I)$ back along this
chain to obtain the piece $\mathcal A(I)$ containing
$\mathbf c_0$. Lemma~\ref{comb:initial-depth-properties}(1) and
Lemma~\ref{new:unmarked-closed} give the required chain. Then
\[
 \mathcal A(I)\xrightarrow{\alpha}I,
\]
where $\alpha=\beta+r_{\mathbf e}(I)$.
Lemma~\ref{comb:initial-depth-properties}(2)--(3) gives
$\mathcal A(I)\Subset\mathcal B(I)\Subset I$.

The first $\beta$ steps from $\mathcal A(I)$ form a restriction
of the chain $\mathcal B(I)\xrightarrow{\beta}I$; by nesting,
this restriction adds no critical-end visits. The remaining
$\alpha-\beta$ steps form the first-return chain of $\mathbf e$
to $I$.
By Lemma~\ref{comb:first-entry}, the entries into $Q$ and $I$
each contribute at most $b_*$ critical-end visits. Hence
\[
 \nu(\mathcal A(I)\to I),\quad\nu(\mathcal B(I)\to I)
 \le2b_*^2+2b_*+1\le3(b_*+1)^2.
\]

\noindent
\textit{Step 4. Exclude $\Post(\mathcal C)$ from the annulus.}
If $\beta>u$, pull $Q'$ back along
\[
 \mathcal B(I)\xrightarrow{\beta-u}Q.
\]
At each critical-end visit, the critical end in the source piece
has its terminal image in $\Post(\mathcal C)\cap Q\subset Q'$.
Thus Lemma~\ref{comb:complete-pullback} applies. The
$\sigma^{\beta-u}$-images of $\mathbf c_0$ and of every end of
$\Post(\mathcal C)$ in $\mathcal B(I)$ also lie in $Q'$.
The lemma and its proof therefore give a piece $Q^*$ such that
\[
 \mathbf c_0\subset Q^*\Subset\mathcal B(I),\qquad
 \sigma^{\beta-u}(Q^*)=Q',\qquad
 \Post(\mathcal C)\cap\mathcal B(I)\subset Q^*.
\]
If $\beta=u$, then $\mathcal B(I)=Q$ and $\mathbf c_0\subset Q'$
because $Q'$ contains the unique critical end of $Q$.
Set $Q^*=Q'$; the same relations hold.

We use the first return time of $\mathbf e$ to $I$ to show that
$Q^*\subset\mathcal A(I)$.
Indeed,
\[
 \mathbf e\subset\sigma^u(Q')\subset I,\qquad
 \sigma^u(Q')\xrightarrow{u'-u}I.
\]
Since $\alpha-\beta$ is the first return time of $\mathbf e$ to $I$,
$\alpha-\beta\le u'-u$.
Hence
\[
 h(Q^*)=h(I)+\beta-u+u'
 \ge h(I)+\alpha=h(\mathcal A(I)).
\]
Since $Q^*$ and $\mathcal A(I)$ contain $\mathbf c_0$, nesting gives
$\Post(\mathcal C)\cap\mathcal B(I)
\subset Q^*\subset\mathcal A(I)$.
\end{proof}

\subsection{Successors}

\begin{definition}\label{def:last-successor}
Let $P$ be a piece containing a critical end in $\mathcal C$.
Denote by $\mathcal D(P)$ its \textbf{last successor}, namely the
successor whose chain to $P$ has the greatest length, and by
$\mathcal D^\tau$ the $\tau$-fold operation.
For a piece $P$ containing an end of $\omega(\mathbf a)$, set
\[
 r(P)=\min\bigl\{j\ge1:\text{there is an end }\mathbf e\subset P
 \text{ with }\sigma^j(\mathbf e)\subset P\bigr\}.
\]
Here $\mathbf e$ need not belong to $\omega(\mathbf a)$.
\end{definition}

The first return of the critical end gives a successor of $P$,
and persistent recurrence gives only finitely many successors.
Thus $\mathcal D(P)$ is well defined. The return time $r(P)$ exists because
every end in $\omega(\mathbf a)$ returns to each piece containing it.

The following fact is used to estimate the lengths of the nest chains.
It follows by the argument of \cite[Lemma~6.4(1)]{RY}; see also
\cite[Lemma~4.2(a), Appendix~A]{PQRTY}.

\begin{lemma}\label{comb:two-successors}
Every piece containing an end of $\mathcal C$ has at least two successors.
\end{lemma}

\begin{proof}
Let $P=P_n(\mathbf c)$, where $\mathbf c\in\mathcal C$, and set
$r=r_{\mathbf c}(P)$ and $Q=L_{\mathbf c}(P)=P_{n+r}(\mathbf c)$.
Lemma~\ref{comb:initial-depth-properties}(1) and
Lemma~\ref{new:unmarked-closed} give the first-entry chains used below.
By Lemma~\ref{comb:first-entry}, $Q$ is a successor of $P$.
Lemma~\ref{comb:initial-depth-properties}(2) gives $Q\Subset P$.
Suppose that $Q$ is the only successor of $P$.

\textit{Step 1. Locate the returns of $\mathbf c$ to $P$.}
We show that, for every $j\ge1$,
\[
 \sigma^j(\mathbf c)\subset P
 \quad\Longrightarrow\quad \sigma^j(\mathbf c)\subset Q.
\]
Otherwise, choose $j\ge1$ with
$\sigma^j(\mathbf c)\subset\mathbb A(Q,P)$,
and set $V=L_{\sigma^j(\mathbf c)}(P)$ and
$s=r_{\sigma^j(\mathbf c)}(P)$. Then $V\subset P$ and
\[
 P_{n+j+s}(\mathbf c)\xrightarrow{j}V\xrightarrow{s}P.
\]
The piece $V$ contains no critical end. Indeed, the only possible
critical end in $P$ is $\mathbf c$, by
Lemma~\ref{comb:initial-depth-properties}(1). If $s\ge r$, nesting
and $\sigma^j(\mathbf c)\not\subset Q$ give $V\cap Q=\varnothing$.
If $s<r$, the inclusion $\mathbf c\subset V$ would give
$\sigma^s(\mathbf c)\subset P$, contrary to the choice of $r$.
Thus $V$ contains no critical end, and nesting also gives
$V\cap Q=\varnothing$ in the second case.

In the chain $P_{n+j+s}(\mathbf c)\xrightarrow{j}V$, let $U$ be
the source piece of its last critical-end visit, and let
$\mathbf d$ be the critical end in $U$. This visit exists because
the initial piece contains $\mathbf c$.
By Lemma~\ref{comb:initial-depth-properties}(1),
$\mathbf d\in\mathcal C\subset\omega(\mathbf c)$.
If $\mathbf d=\mathbf c$, start the chain at $U$.
Otherwise, prepend the first-entry chain of $\mathbf c$ into $U$,
which exists since $\mathbf d\in\omega(\mathbf c)$.
Append the chains from $U$ to $V$ and from $V$ to $P$.

By Lemma~\ref{comb:first-entry}, each critical end occurs at most
once before the terminal piece of each first-entry chain.
The prepended chain does not meet $\mathbf d$ before $U$;
the chain from $U$ to $V$ has exactly one critical-end visit;
and the chain from $V$ to $P$ meets $\mathbf c$ only at $P$.
Thus each critical end occurs at most twice in the resulting
chain, including both endpoints, so its initial piece is a
successor of $P$. The orbit of $\mathbf c$ enters $V\subset P$
at a positive time, which is at least $r$, and then returns to
$P$ after another $s>0$ iterates. The new chain therefore has
length greater than $r$, contradicting the uniqueness of $Q$.

\noindent
\textit{Step 2. Pull back to deeper pieces.}
We prove by induction on $m\ge1$ that
\[
 \sigma^j(\mathbf c)\subset P
 \quad\Longrightarrow\quad
 \sigma^j(\mathbf c)\subset P_{n+mr}(\mathbf c)
 \qquad(j\ge1).
\]
The case $m=1$ is Step~1. Suppose the assertion holds for $m$.
We verify the critical-end condition of
Lemma~\ref{comb:complete-pullback} for the chain
$Q\xrightarrow{r}P$ and the inner piece $P_{n+mr}(\mathbf c)$.
Let $\mathbf c'\subset\sigma^i(Q)$ be a critical end, with
$0\le i<r$. By Lemma~\ref{comb:initial-depth-properties}(1),
$\mathbf c'\in\mathcal C\subset\omega(\mathbf c)$.
Choose iterates of $\mathbf c$ in $\sigma^i(Q)$ converging to
$\mathbf c'$. After another $r-i$ iterates, they return to $P$
and hence lie in $P_{n+mr}(\mathbf c)$ by induction.
Since ends in $\omega(\mathbf a)$ avoid the boundary of this
piece, those contained in it form a closed subfamily of
$\omega(\mathbf a)$. Continuity of $\sigma$ therefore gives
\[
 \sigma^{r-i}(\mathbf c')\subset P_{n+mr}(\mathbf c).
\]

Since $P_{n+mr}(\mathbf c)\subset Q\Subset P$,
Lemma~\ref{comb:complete-pullback} gives a unique piece in $Q$
whose $\sigma^r$-image is $P_{n+mr}(\mathbf c)$.
The induction assumption at $j=r$ gives
$\sigma^r(\mathbf c)\subset P_{n+mr}(\mathbf c)$, so the proof
of that lemma shows that this piece contains $\mathbf c$.
Its depth is $n+(m+1)r$, so it is $P_{n+(m+1)r}(\mathbf c)$.
Now let $j\ge1$ satisfy $\sigma^j(\mathbf c)\subset P$.
Step~1 gives $\sigma^j(\mathbf c)\subset Q$, so
$\sigma^{j+r}(\mathbf c)\subset P$.
The induction assumption at $j+r$ gives
$\sigma^{j+r}(\mathbf c)\subset P_{n+mr}(\mathbf c)$.
The proof of Lemma~\ref{comb:complete-pullback} also gives
$\sigma^j(\mathbf c)\subset P_{n+(m+1)r}(\mathbf c)$.
This completes the induction.

\noindent
\textit{Step 3. Apply puzzle separation.}
Taking $j=r$ in Step~2, we find that $\mathbf c$ and
$\sigma^r(\mathbf c)$ lie in $P_{n+mr}(\mathbf c)$ for every
$m\ge1$. Lemma~\ref{dyn:puzzle}(2) gives
$\sigma^r(\mathbf c)=\mathbf c$, contrary to the fact that
$\mathbf c$ is wandering.
\end{proof}

\subsection{Construction and estimates}

\begin{proof}[Proof of Proposition~\ref{comb:enhanced-nest}]
Fix $\mathbf c_0$ and $K_0$ as in the statement. For $n\ge1$, set
\[
 \tau=b_*+5,\qquad I_n=\mathcal D^\tau(K_{n-1}),\qquad
 K_n=\mathcal A(I_n),\quad K'_n=\mathcal B(I_n).
\]
Let $p_n,p'_n$ be the lengths of the chains from $K_n,K'_n$
to $K_{n-1}$, respectively. Then $p_n>p'_n>0$ and
$h(K_n,K'_n)=p_n-p'_n$.

For (1), nesting follows from Proposition~\ref{comb:AB} and
Lemma~\ref{comb:initial-depth-properties}(2); the annulus avoids
$\Post(\mathcal C)$ by Proposition~\ref{comb:AB}.

For (2), each successor contributes at most $2b_*$ critical-end visits,
so either chain has at most $3(b_*+1)^2+2b_*\tau\le N_0$ visits,
and Lemma~\ref{ana:real-degree} gives the degree bound.

For (3), let $P'\xrightarrow{s}P$ be a chain of length $s>0$ with
$\mathbf c\subset P'\cap P$ for some $\mathbf c\in\mathcal C$.
We first show that
\[
 r(P')\ge r(P),\qquad
 r(P)\le s\le\nu(P'\to P)\,r(P').
\]
Nesting gives $P'\subset P$, hence the first inequality.
The lower bound on $s$ follows from
$\mathbf c,\sigma^s(\mathbf c)\subset P$.
For the upper bound, list the times
$0=t_0<t_1<\cdots<t_a=s$ at which
$\mathbf c\subset\sigma^{t_i}(P')$. Then
$a\le\nu(P'\to P)$.
For $0\le i<a$, set $Q=\sigma^{t_i}(P')$.
If an end $\mathbf e\subset Q$ returned to $Q$ at a time
$0<j<t_{i+1}-t_i$, then
$\sigma^j(\mathbf e)\subset Q\cap\sigma^j(Q)$.
Nesting would give $Q\subset\sigma^j(Q)$ and hence
$\mathbf c\subset\sigma^j(Q)$, contrary to the choice of
the consecutive times $t_i,t_{i+1}$. Thus
$t_{i+1}-t_i\le r(Q)\le r(P')$.
Adding these inequalities proves the upper bound.

Now let $s$ be the length of the last successor chain
$\mathcal D(P)\xrightarrow{s}P$. We show that
\[
 2r(P)\le s\le r(\mathcal D(P)).
\]
If an end in $\mathcal D(P)$ returned to $\mathcal D(P)$ at a
time $0<j<s$, nesting
would give $\mathcal D(P)\subset\sigma^j(\mathcal D(P))$ and hence
$\mathbf c\subset\sigma^j(\mathcal D(P))$, contrary to the successor
condition. Thus $s\le r(\mathcal D(P))$.
By Lemma~\ref{comb:two-successors},
$s>k=r_{\mathbf c}(P)$. Since
$\mathbf c,\sigma^k(\mathbf c),\sigma^s(\mathbf c)\subset P$,
both $k$ and $s-k$ are at least $r(P)$, giving $s\ge2r(P)$.

By construction, $p_n=\alpha(I_n)+h(I_n,K_{n-1})$.
Since $r(\mathcal D(P))\ge2r(P)$, we have
$r(I_n)\ge2^\tau r(K_{n-1})$. The ends $\mathbf c_0$ and
$\sigma^{\beta(I_n)}(\mathbf c_0)$ lie in $I_n$ and return there
after $\beta(I_n)$ and $\alpha(I_n)-\beta(I_n)$ iterates,
respectively. Both times are at least $r(I_n)$, so
applying $s\le\nu(P'\to P)r(P')$ to
$K_n\xrightarrow{\alpha(I_n)}I_n$ and using
Proposition~\ref{comb:AB} gives
\[
 2^{\tau+1}r(K_{n-1})\le2r(I_n)\le\alpha(I_n)
 \le3(b_*+1)^2r(K_n).
\]
Summing $2r(P)\le s\le r(\mathcal D(P))$ over the $\tau$
successive last-successor chains gives
\[
 (2^{\tau+1}-2)r(K_{n-1})\le h(I_n,K_{n-1})
 <2r(I_n)\le2r(K_n).
\]
Adding the bounds for $\alpha(I_n)$ and $h(I_n,K_{n-1})$ gives
\[
 (2^{\tau+2}-2)r(K_{n-1})\le p_n
 \le\bigl(3(b_*+1)^2+2\bigr)r(K_n).
\]
The lower bound for $p_{n+1}$ and the upper bound for $p_n$ give
\[
 p_{n+1}\ge
 \frac{2^{\tau+2}-2}{3(b_*+1)^2+2}\,p_n\ge2p_n,
\]
since $\tau=b_*+5$ and $b_*\ge1$. Therefore
$s(n,0)\le p_n(1+2^{-1}+\cdots+2^{-(n-1)})\le2p_n$.

It remains to prove $h(K_{n+2},K'_{n+2})>s(n,0)$.
The end $\sigma^{\beta(I_n)}(\mathbf c_0)\subset I_n$ returns
to $I_n$ after $\alpha(I_n)-\beta(I_n)=h(K_n,K'_n)$ steps.
Thus $h(K_n,K'_n)\ge r(I_n)$. Using the estimates above and
$r(K_{n+1})\ge r(K_n)$, we obtain
\[
 h(K_{n+2},K'_{n+2})\ge2^\tau r(K_{n+1})
 >2\bigl(3(b_*+1)^2+2\bigr)r(K_n)
 \ge2p_n\ge s(n,0),
\]
where the strict inequality follows from $\tau=b_*+5$ and $b_*\ge1$.
\end{proof}

\par
\begingroup
\footnotesize
\noindent\textsc{Jinsong Zeng}\\
School of Mathematical Sciences, Shenzhen University\\
Shenzhen 518060, People's Republic of China\\
\textit{Email:} \href{mailto:jinsongzeng@163.com}{\nolinkurl{jinsongzeng@163.com}}
\par\endgroup


\begin{thebibliography}{PQRTY}

\bibitem[Ahl]{Ahlfors}
L.~V.~Ahlfors,
\emph{Conformal invariants: Topics in geometric function theory},
McGraw-Hill Series in Higher Mathematics,
McGraw-Hill, New York, 1973.

\bibitem[BH]{BH}
B.~Branner and J.~H.~Hubbard,
\emph{The iteration of cubic polynomials. Part II: Patterns and parapatterns},
Acta Mathematica \textbf{169} (1992), 229--325.
\href{https://doi.org/10.1007/BF02392761}{doi:10.1007/BF02392761}.

\bibitem[BM]{BM}
M.~Bonk and D.~Meyer,
\emph{Expanding Thurston Maps},
Mathematical Surveys and Monographs, vol.~225,
American Mathematical Society, Providence, RI, 2017.

\bibitem[CP]{CP}
G.~Cui and W.~Peng,
\emph{On the structure of Fatou domains},
Science in China Series A: Mathematics \textbf{51} (2008), 1167--1186.
\href{https://doi.org/10.1007/s11425-008-0056-5}{doi:10.1007/s11425-008-0056-5}.

\bibitem[CPTa]{CPT}
G.~Cui, W.~Peng, and L.~Tan,
\emph{On the topology of wandering Julia components},
Discrete and Continuous Dynamical Systems \textbf{29} (2011), 929--952.
\href{https://doi.org/10.3934/dcds.2011.29.929}{doi:10.3934/dcds.2011.29.929}.

\bibitem[CPTb]{CPT16}
G.~Cui, W.~Peng, and L.~Tan,
\emph{Renormalizations and wandering Jordan curves of rational maps},
Communications in Mathematical Physics \textbf{344} (2016), 67--115.
\href{https://doi.org/10.1007/s00220-016-2623-x}{doi:10.1007/s00220-016-2623-x}.

\bibitem[CPY]{CPY}
G.~Cui, W.~Peng, and L.~Yang,
\emph{Wandering Julia components of cubic rational maps},
Mathematische Zeitschrift \textbf{308} (2024), no.~4, Paper No.~69.
\href{https://doi.org/10.1007/s00209-024-03629-y}{doi:10.1007/s00209-024-03629-y}.

\bibitem[DS]{DS}
K.~Drach and D.~Schleicher,
\emph{Rigidity of Newton dynamics},
Advances in Mathematics \textbf{408} (2022), 108591.
\href{https://doi.org/10.1016/j.aim.2022.108591}{doi:10.1016/j.aim.2022.108591}.

\bibitem[GGPa]{GGP1}
N.~Gao, Y.~Gao, and W.~Peng,
\emph{On the parabolic Fatou domains},
\href{https://arxiv.org/abs/2509.09914}{arXiv:2509.09914} (2025).

\bibitem[GGPb]{GGP2}
N.~Gao, Y.~Gao, and W.~Peng,
\emph{Rigidity of simple parabolic maps},
\href{https://arxiv.org/abs/2608.22328v2}{arXiv:2608.22328v2} (2026).

\bibitem[GXY]{GXY}
Y.~Gao, L.~Xu, and L.~Yang,
\emph{Non-recurrent rational maps with disconnected Julia set},
Nonlinearity \textbf{39} (2026), 045004.
\href{https://doi.org/10.1088/1361-6544/ae54f4}{doi:10.1088/1361-6544/ae54f4}.

\bibitem[God]{Godillon}
S.~Godillon,
\emph{A family of rational maps with buried Julia components},
Ergodic Theory and Dynamical Systems \textbf{35} (2015), no.~6, 1846--1879.
\href{https://doi.org/10.1017/etds.2014.22}{doi:10.1017/etds.2014.22}.

\bibitem[KL]{KL}
J.~Kahn and M.~Lyubich,
\emph{The Quasi-Additivity Law in conformal geometry},
Annals of Mathematics (2) \textbf{169} (2009), 561--593.
\href{https://doi.org/10.4007/annals.2009.169.561}{doi:10.4007/annals.2009.169.561}.

\bibitem[KS]{KS}
O.~Kozlovski and S.~van Strien,
\emph{Local connectivity and quasi-conformal rigidity of non-renormalizable polynomials},
Proceedings of the London Mathematical Society (3) \textbf{99} (2009), 275--296.
\href{https://doi.org/10.1112/plms/pdn055}{doi:10.1112/plms/pdn055}.

\bibitem[KSS]{KSS}
O.~Kozlovski, W.~Shen, and S.~van Strien,
\emph{Rigidity for real polynomials},
Annals of Mathematics (2) \textbf{165} (2007), 749--841.
\href{https://doi.org/10.4007/annals.2007.165.749}{doi:10.4007/annals.2007.165.749}.

\bibitem[McMa]{McMullen88}
C.~T.~McMullen,
\emph{Automorphisms of rational maps},
in \emph{Holomorphic Functions and Moduli I},
Mathematical Sciences Research Institute Publications, vol.~10,
Springer, New York, 1988, pp.~31--60.
\href{https://doi.org/10.1007/978-1-4613-9602-4_3}{doi:10.1007/978-1-4613-9602-4\_3}.

\bibitem[McMb]{McMullen}
C.~T.~McMullen,
\emph{Complex Dynamics and Renormalization},
Annals of Mathematics Studies, vol.~135,
Princeton University Press, Princeton, NJ, 1994.

\bibitem[Mil]{Milnor}
J.~Milnor,
\emph{Dynamics in One Complex Variable}, third edition,
Annals of Mathematics Studies, vol.~160,
Princeton University Press, Princeton, NJ, 2006.

\bibitem[Moo]{Moore}
R.~L.~Moore,
\emph{Concerning upper semi-continuous collections of continua},
Transactions of the American Mathematical Society \textbf{27} (1925),
416--428.
\href{https://doi.org/10.1090/S0002-9947-1925-1501320-8}{doi:10.1090/S0002-9947-1925-1501320-8}.

\bibitem[PQRTY]{PQRTY}
W.~Peng, W.~Qiu, P.~Roesch, L.~Tan, and Y.~Yin,
\emph{A tableau approach of the KSS nest},
Conformal Geometry and Dynamics \textbf{14} (2010), 35--67.
\href{https://doi.org/10.1090/S1088-4173-10-00201-8}{doi:10.1090/S1088-4173-10-00201-8}.

\bibitem[PT]{PT}
K.~Pilgrim and L.~Tan,
\emph{Rational maps with disconnected Julia set},
Ast\'erisque \textbf{261} (2000), 349--384.
\href{https://www.numdam.org/item/AST_2000__261__349_0/}{numdam.org}.

\bibitem[QY]{QY}
W.~Qiu and Y.~Yin,
\emph{Proof of the Branner--Hubbard conjecture on Cantor Julia sets},
Science in China Series A: Mathematics \textbf{52} (2009), 45--65.
\href{https://doi.org/10.1007/s11425-008-0178-9}{doi:10.1007/s11425-008-0178-9}.

\bibitem[RY]{RY}
P.~Roesch and Y.~Yin,
\emph{Bounded critical Fatou components are Jordan domains for polynomials},
Science China Mathematics \textbf{65} (2022), 331--358.
\href{https://doi.org/10.1007/s11425-020-1827-4}{doi:10.1007/s11425-020-1827-4}.

\bibitem[RYZ]{RYZ}
P.~Roesch, Y.~Yin, and J.~Zeng,
\emph{Rigidity of non-renormalizable Newton maps},
Science China Mathematics \textbf{67} (2024), 855--872.
\href{https://doi.org/10.1007/s11425-022-2136-1}{doi:10.1007/s11425-022-2136-1}.

\bibitem[Sul]{Sullivan}
D.~Sullivan,
\emph{Quasiconformal homeomorphisms and dynamics I: Solution of the Fatou--Julia problem on wandering domains},
Annals of Mathematics (2) \textbf{122} (1985), 401--418.
\href{https://doi.org/10.2307/1971308}{doi:10.2307/1971308}.

\bibitem[WYZ]{WYZ}
X.~Wang, Y.~Yin, and J.~Zeng,
\emph{Dynamics of Newton maps},
Ergodic Theory and Dynamical Systems \textbf{43} (2023), 1035--1080.
\href{https://doi.org/10.1017/etds.2021.168}{doi:10.1017/etds.2021.168}.

\bibitem[Zen]{ZengFJ}
J.~Zeng,
\emph{Most Fatou and Julia components are small for polynomials},
Nonlinearity \textbf{39} (2026), no.~7, 075012.
\href{https://doi.org/10.1088/1361-6544/ae8214}{doi:10.1088/1361-6544/ae8214}.

\end{thebibliography}
\end{document}